%% file: These_13Julio_.tex
\documentclass[10pt]{amsart}
\usepackage{amsmath,amsfonts,amsthm,latexsym,amssymb,amscd,amsxtra,accents,mathrsfs}
\usepackage [dvips]{graphicx}
\theoremstyle{definition}\newtheorem{Definition}{Definition}[section]
\theoremstyle{definition}
\theoremstyle{plain}\newtheorem{Lemma}{Lemma}[section]
\theoremstyle{plain}\newtheorem{Theorem}{Theorem}[section]
\theoremstyle{plain}\newtheorem{Proposition}{Proposition}[section]
\theoremstyle{plain}\newtheorem{Corollary}{Corollary}[section]
\theoremstyle{remark}\newtheorem{Remark}{Remark}[section]
\theoremstyle{plain}\newtheorem{Affirmation}{Affirmation}
\theoremstyle{plain}
\theoremstyle{plain}\newtheorem{Claim}{Claim}

\title[]
{Robust Transitivity and Density of Periodic Points of Partially Hyperbolic Diffeomorphisms}
\author[Alien Herrera Torres \hspace{0,2cm} and \hspace{0,2cm} Ana T\'ercia Monteiro Oliveira]{Alien Herrera Torres \hspace{0,2cm} and \hspace{0,2cm} Ana T\'ercia Monteiro Oliveira}
\date{\today}

\makeindex
\begin{document}

\begin{abstract}
We prove results related to robust transitivity and density of periodic points of Partially 
Hyperbolic Diffeomorphisms under conditions involving
Accessibility and a property in the tangent bundle .
\end{abstract}

\maketitle

\tableofcontents

\newpage

\input{Introduction}

\input{Preliminaries}

\input{PropertySHandTopologicallyMixing}

\input{PropertySHandAccessibility}

\input{PropertySHandRobustTransitivity}

\input{PropertySHandDensityofPeriodicPoints}

\input{Examples}

\bibliographystyle{plain}

\bibliography{References(Viana)}

\printindex
\end{document}

%% file: Introduction.tex
\section{Introduction}
\label{sec:Introduction}
\noindent A very interesting feature of a differentiable dynamical system is topological transitivity. 
Being a sign of complexity of the underlying 
dynamics it prevents the possibility of reducing its study to simplest systems.

One of the most important questions in the theory of Differentiable Dynamical Systems
regarding a particular dynamical property is to recognize when it is present in all 
nearby systems (with respect to some topology). When this happens we say that the 
property is robust or stable under perturbations.

So the search for conditions on a differentiable dynamical system leading to
robust transitivity has been a topic of interest for a long time. Many examples
exhibiting robust transitivity has been studied, beginning with the transitive 
Anosov diffeomorphisms.

Robust transitivity is not an exclusive property of Hyperbolic Diffeomorphisms
as it has been showed first by the example of Shub on the torus $\mathbb{T}^{4}$, 
later by the example of Ma\~{n}\'e on the torus $\mathbb{T}^{3}$ and more recently
by the example of Bonatti and D\'ias in \cite{BD}. All these examples are Partially
Hyperbolic Systems (see section~\ref{sec:Preliminaries}). While other example,
due to Bonatti and Viana \cite{BV}, exhibits just dominated splitting.

Ergodicity and its stability are other important properties
to study on a dynamical system. A well known conjecture formulated by Pugh and Shub 
\cite{PuSh} on Stable Ergodicity for Partially Hyperbolic Systems has been the motivation
for a lot of research during the last few years.

One of the conditions appearing on the hypothesis of this conjecture is that of
Accessibility (see section~\ref{sec:Preliminaries}) which as the work \cite{DW}
shows is a typical property in the sense that it is $C^{1}$ dense among 
the $C^{r}$ Partially Hyperbolic Diffeomorphisms of a compact manifold. 

Accessibility has also a relation with transitivity according to Brin's Theorem \cite{P}
stating that in a Partially Hyperbolic Accessible System, transitivity is equivalent 
to the fact of the non-wandering set being the whole manifold.

In connection with the mentioned examples of Shub and Ma\~{n}\'e, the authors
Pujals and Sambarino introduced in \cite{PS} an interesting property which they 
call Property SH,  to guarantee that
the strong stable foliation is robustly minimal. A key feature of Property SH is its
intrinsic robustness which makes it an appealing condition to use for 
establishing robust transitivity in more general contexts.

This work has been motivated by the idea of exploring the consequences, in the
sense of robust transitivity, of the combination of  
Property SH and Accessibility, for Partially Hyperbolic Systems.

Our first result arise naturally 
after the observation that in the proof of Brin's Theorem the accessibility 
in relation to open sets 
(as defined in Section~\ref{sec:Preliminaries}) is enough to guarantee 
transitivity. Having at hand the Property SH and its robustness we can 
establish the robustness of accessibility in relation to open sets, 
see Corollary~\ref{Corollary:Robustness of accesibility by open sets}. 
As a Corollary, it follows our first result 
(see Section~\ref{sec:Property SH and Accessibility}) related to 
robust transitivity:

\begin{Corollary}\label{Corollary:A}
Let $f\in Diff^{r}(M)$ be a partially hyperbolic, accessible, volume preserving diffeomorphism
exhibiting the property SH. Then any diffeomorphism $C^{1}$-close to $f$ that is volume
preserving  is topologically mixing.
\end{Corollary}

Abdenur and Crovisier proved in \cite{AC} that the fact of a 
diffeomorphism being robustly transitive implies that it is also
topologically mixing modulo an arbitrarily small $C^{1}$ perturbation.
In the same spirit  we have the following Theorem 
in Section~\ref{sec:Property SH and Topologically Mixing}:

\begin{Theorem}\label{Theorem:B}
Let $f\in Diff^{r}(M)$ be a partially hyperbolic diffeomorphism, accessible in relation to open sets, 
satisfying $\Omega (f)=M$ and the Property SH, then $f$ is topologically mixing.
\end{Theorem}

\noindent In section~\ref{sec:Examples} we show that Shub's example in
$\mathbb{T}^{4}$ satisfies all the conditions in Theorem~\ref{Theorem:B}.

\noindent Theorem~\ref{Theorem:C} by Miss Oliveira proves that Property SH is enough 
to guarantee robust transitivity. 

\begin{Theorem}\label{Theorem:C}
Let $M$ be a compact Riemannian manifold and let  $f\in Diff^{r}(M)$ be a partially 
hyperbolic diffeomorphism, non-hyperbolic, transitive. If $f$ and $f^{-1}$ satisfy 
Property SH then $f$ is robustly transitive.
\end{Theorem}

The point in this Theorem is that the hypotheses are
given only on the tangent bundle. The condition of minimality
of the stable foliation assumed in \cite{PS} was substituted
here by the Property SH for $f^{-1}$. In 
section~\ref{sec:Examples} we give an scenario where all
the conditions in Theorem~\ref{Theorem:C} are realized.
An interesting question not addressed in our work is if it is 
possible to exploit Theorem~\ref{Theorem:C} to produce new examples
of robustly transitive diffeomorphisms.

Finally we would like to express our deep gratitude to professor Enrique Pujals
from IMPA for the multiple suggestions and clarifying discussions during
the course of our work.

In the following sections $M$ will denote a compact Riemannian manifold and
$Diff^{r}(M)$ the set of $C^{r}$-diffeomorphims defined on $M$.

%% file: Preliminaries.tex
\section{Preliminaries}
\label{sec:Preliminaries}
\noindent In this section we recall some well-known results regarding partially hyperbolic systems. 
We refer to \cite{HPS}, \cite{P}, \cite{S}, \cite{PS} for a general background on the topics we will review.

\subsection{Partially Hyperbolic Diffeomorphisms}
\label{subsec:Partially Hyperbolic Diffeomorphisms}\quad
\newline
\noindent As Property SH plays a central role in this work we follow closely the definitions
and basic results in \cite{PS}.

\begin{Definition}\label{Definition:Partially Hyperbolic Diffeomorphism}
A diffeomorphism $f:M\rightarrow M$ is partially hyperbolic  
provided the tangent bundle splits into three non-trivial sub-bundles $TM=E^{ss}\oplus E^{c}\oplus E^{uu}$
which are invariant under the tangent map $Df$ and there are $0<\lambda<\mu<1$ 
such that for all $x\in M$ 
\[\parallel Df_{\mid E^{ss}(x)}\parallel<\lambda,\hspace{0.4cm}  \parallel Df^{-1}_{\mid E^{uu}(x)}\parallel<\lambda,\hspace{0.4cm} \mu<\parallel Df^{-1}_{\mid E^{c}(x)}\parallel,\hspace{0.4cm}  \mu<\parallel Df_{\mid E^{c}(x)}\parallel.\]
\end{Definition}

\begin{Lemma}\label{Lemma:Partial hyperbolicity is open}
Let $f\in Diff^{r}(M)$ be a partially hyperbolic diffeomorphism and $Gr _{k}(M)$
denote the Grassmannian bundle of $M$ of $k$-dimensional spaces. 
Then there exist a $C^{r}$ neighbourhood of $f$, say $\mathcal{U}$, numbers $\lambda_{1}$ and $\mu_{1}$ with
$0<\lambda<\lambda_{1}<\mu<\mu_{1}<1$ and continuous functions $E^{ss}:\mathcal{U}\rightarrow C(M,Gr_{dim (E^{ss})}(M))$, $E^{c}:\mathcal{U}\rightarrow C(M,Gr_{dim (E^{c})}(M))$ and $E^{uu}:\mathcal{U}\rightarrow C(M,Gr_{dim (E^{uu})}(M))$ such that, for any $g\in\mathcal{U}$ and $x\in M$, we have the following:\\
\newline
(1) $TM=E^{ss}(g)(M)\oplus E^{c}(g)(M)\oplus E^{uu}(g)(M)$, this decomposition is invariant under $Dg$ and no one of these sub-bundles is trivial\\
\newline
\noindent (2) $\parallel Dg_{\mid E^{ss}(x,g)}\parallel<\lambda_{1}$, $\parallel Dg^{-1}_{\mid E^{uu}(x,g)}\parallel<\lambda_{1}$\\
\newline
\noindent (3) $\mu_{1}<\parallel Dg^{-1}_{\mid E^{c}(x,g)}\parallel$, $\mu_{1}<\parallel Dg_{\mid E^{c}(x,g)}\parallel$
\end{Lemma}
\noindent The sub-bundles $E^{ss}(g)(M)$ and $E^{uu}(g)(M)$ are uniquely integrable and form two foliations $\mathcal{F}^{ss}$ and $\mathcal{F}^{uu}$.

\begin{Theorem}\label{Theorem:Continous dependence of foliations}
Let $\mathcal{U}$ be as in Lemma~\ref{Lemma:Partial hyperbolicity is open}. 
Then, for each $g\in\mathcal{U}$ there are two partitions $\mathcal{F}^{ss}(g)$ 
and $\mathcal{F}^{uu}(g)$ of $M$ such that for each $x\in M$ the elements of 
the partitions that contain $x$, denoted by $\mathcal{F}^{ss}(x,g)$ and 
$\mathcal{F}^{uu}(x,g)$ are $C^{1}$ submanifolds such that $T_{x}\mathcal{F}^{ss}(x,g)=E^{ss}(x,g)$ 
and $T_{x}\mathcal{F}^{uu}(x,g)=E^{uu}(x,g)$. These submani\-folds depend 
continuously (on compact subsets) on $x\in M$ and $g\in\mathcal{U}$. 
\end{Theorem}
These submanifolds 
$\mathcal{F}^{ss}(x,g)$ and $\mathcal{F}^{uu}(x,g)$ inherit the Riemannian metric on $M$. 
We shall denote by $\mathcal{F}^{ss}_{r}(x,g)$ (respectively $\mathcal{F}^{uu}_{r}(x,g)$) 
the ball in $\mathcal{F}^{ss}(x,g)$ (respectively $\mathcal{F}^{uu}(x,g)$) of radius $r$ centred at $x$.

The sub-bundle $E^{cu}=E^{c}\oplus E^{uu}$ is not integrable in general. However, we can choose a continuous family of locally invariant manifolds tangent to it. Let $dim E^{cu}=l$ and denote by $I_{\epsilon}$ the ball of radius $\epsilon$ in $R^{l}$.

\begin{Lemma}\label{Lemma:Inclusion of backward iteration}
Let $\mathcal{U}$ be as in Lemma~\ref{Lemma:Partial hyperbolicity is open}
and $Emb_{1} (I_{1},M)$ the set of $C^{1}$-embeddings of $I_{1}$ in $M$.
There exists a continuous map
$\varphi:M\times\mathcal{U} \rightarrow Emb_{1} (I_{1},M)$ such that, if we set 
$W^{cu}_{\epsilon}(x,g)=\varphi (x,g) I_{\epsilon}$, then the following hold:\\
\newline
\noindent (1) $T_{x} W^{cu}_{\epsilon}(x,g)=E^{cu}(x,g)$;\\
\newline
\noindent (2) given $\epsilon>0$ there exists $r=r(\epsilon)>0$ 
such that $g^{-1}(W^{cu}_{r}(x,g))\subset W^{cu}_{\epsilon}(g^{-1}(x),g).$
\end{Lemma}
\noindent For the sake of simplicity we shall identify $W^{cu}_{\epsilon}(x,g)$ 
with the ball of radius $\epsilon$ in $W^{cu}_{1}(x,g)$.
\begin{Lemma}\label{Lemma:Backward iterations contract}
Let $\mathcal{U}$ be as in Lemmas~\ref{Lemma:Partial hyperbolicity is open} and 
~\ref{Lemma:Inclusion of backward iteration}. 
Given $0<\lambda<\lambda_{1}<1$ there exists $r_{0}$ such that if $g\in\mathcal{U}$ and $x\in M$ satisfy
\[
\prod_{j=0}^{n}\parallel Dg^{-1}_{\mid_{E^{cu}(g^{-j}(x)})}\parallel<\lambda^{n},\hspace{1 cm} 0\leq n\leq m,
\]
Then $g^{-m}(W^{cu}_{r_{0}}(x,g))\subset W^{cu}_{\lambda^{m}_{1}r_{0}}(g^{-m}(x),g)$.
\end{Lemma}

\noindent In the following, we will work with partially hyperbolic diffeomorphisms.
\newpage

\subsection{Accessibility}
\label{subsec:Accessibility} \quad
\newline

\begin{Definition}\label{Definition:Accesible pairs of points}
Let $f\in Diff^{r}(M)$ be a partially hyperbolic diffeomorphism.
Two points $p,q\in M$ are called accessible, if there are points 
$z_{0}=p, z_{1},\dots, z_{l-1}, z_{l}=q, z_{i}\in M$, such that 
$z_{i}\in\mathcal{F}^{\alpha}(z_{i-1},f)$ for $i=1,\dots,l$ and $\alpha =ss$ or $uu$.
\end{Definition}
\noindent The collection of points $z_{0}, z_{1}, \dots , z_{l}$ is called the $us$-path 
connecting $p$ and $q$. Accessibility is an equivalence relation and the collection of points 
accessible from a given point $p$ is called the accessibility class of $p$.
We will denote this class by $\mathcal{C}(p,f)$. The diffeomorphism $f$ is said to have the accessibility property 
if the accessibility class of any point is 
the whole manifold $M$, or, in other words, if  
any two points in $M$ are accessible.

\noindent Next we introduce the notion of accessibility in relation to open sets,
which we use to give a stronger version of Brin's Theorem (See
section~\ref{sec:Property SH and Topologically Mixing}).

\begin{Definition}\label{Definition:Accesible pairs of open sets}
Two open sets $P, Q\subseteq M$ are called accessible, 
if there are points $p\in P$, $q\in Q$, 
such that $p, q$ are accessible.
\end{Definition}
\noindent We will call a diffeomorphism $f$ accessible in relation to open sets 
if any two open sets are accessible.\\

\noindent Obviously accessibility implies accessibility in relation to 
open sets. The converse is not true in general. It is worth noting that
if $f$ is accessible (respectively accessible in relation to open sets) then
$f^{n}$ is accessible (respectively accessible in relation to open sets)
for any $n\in\mathbb{Z}$.

\begin{Remark}
It is not difficult to prove that accessibility in relation to open sets
is equivalent to the existence of a residual set $R$ in $M$ whose points
have a dense accessibility class. Indeed if $\{U_{k}\}_{k\in\mathbb{N}}$
is a countable base of open sets of the manifold and $C_{k}$ is the set 
of points accessible to at least one point in $U_{k}$, we can take $R=\cap_{k\in\mathbb{N}}C_{k}$
as such a residual set.
\end{Remark}

\begin{Lemma}\label{Lemma:Delta accessibility}
Assume that a partially hyperbolic diffeomorphism $f$ has the accessibility property. 
Then for every $\delta>0$ there exist $l>0$ and $R>0$ such that for any $p, q\in M$ 
one can find a $us$-path that starts at $p$, ends within distance $\frac{\delta}{2}$ 
of $q$, and has at most $l$ legs, each of them with length at most $R$.
\end{Lemma}
\begin{proof}
See \cite{P}.
\end{proof}

\begin{Lemma}\label{Lemma:A kind of continuity of accessible paths}
Let $f: M\rightarrow M$ be a partially hyperbolic accessible diffeomorphism. 
Given $p_{0}\in M$, there is $q_{0}\in M$ and a $us$-path 
$z_{0}(q_{0})=p_{0}, z_{1}(q_{0}),\dots , z_{N}(q_{0})=q_{0}$ 
connecting $p_{0}$ to $q_{0}$ and satisfying the following 
property: for any $\epsilon >0$ there exist $\delta >0$ and $L>0$ 
such that for every $x\in B(q_{0},\delta)$ there exists a $us$-path 
$z_{0}(x)=p_{0}, z_{1}(x), \dots ,z_{N}(x)=x$ connecting $p_{0}$ to $x$ and 
such that $dist (z_{j}(x), z_{j}(q_{0}))<\epsilon$ and $dist_{\mathcal{F}^{\alpha}}(z_{j-1}(x), z_{j}(x))< L$ 
for $j=1, \dots ,N$ where $dist_{\mathcal{F}^{\alpha}}$ denotes the distance along the 
strong (either stable or unstable) leaf common to the two points.
\end{Lemma}
\begin{proof}
See \cite{S}.
\end{proof}

\begin{Lemma}\label{Lemma:Number and lengths of paths legs can be limited}
Assume that a partially hyperbolic diffeomorphism $f$ has the accessibility property. 
Then there exist $l_{0}>0$ and $R_{0}>0$ such that for any $p, q \in M$ one can find 
a $us$-path that starts at $p$, ends at $q$, and has at most $l_{0}$ legs, each of them 
with length at most $R_{0}$.
\end{Lemma}
\begin{proof}
Fix $p_{0}\in M$. Let $q_{0}\in M$ and a $us$-path 
$z_{0}(q_{0})=p_{0}, z_{1}(q_{0}), \dots ,z_{N}(q_{0})=q_{0}$ be as 
in Lemma~\ref{Lemma:A kind of continuity of accessible paths}. Let $\epsilon >0$. Take 
$\delta >0$ and $L>0$ as in Lemma~\ref{Lemma:A kind of continuity of accessible paths}. For 
this $\delta >0$ take $l>0$ and $R>0$ as in Lemma~\ref{Lemma:Delta accessibility}. Next, set 
$l_{0}=2l+2N$ and $R_{0}=max \lbrace R, L\rbrace$. Let $p, q\in M$. From 
Lemma~\ref{Lemma:Delta accessibility} we know that there exists a $us$-path that starts 
at $p$ (respectively $q$), ends within distance 
$\delta$ of $q_{0}$, say at $p_{1}$ (respectively $q_{1}$), and has at most $l$ legs, each 
of them with length at most $R$.\\
From Lemma~\ref{Lemma:A kind of continuity of accessible paths} there exist 
a $us$-path $z_{0}(p_{1})=p_{0}, z_{1}(p_{1}), \dots ,z_{N}(p_{1})=p_{1}$ connecting $p_{0}$ 
to $p_{1}$ and a $us$-path $z_{0}(q_{1})=p_{0}, z_{1}(q_{1}), \dots ,z_{N}(q_{1})=q_{1}$ 
connecting $p_{0}$ to $q_{1}$.\\
Thus,
\[
p_{1}=z_{N}(p_{1}), z_{N-1}(p_{1}), \dots, z_{0}(p_{1})=p_{0}=z_{0}(q_{1}), z_{1}(q_{1}), \dots ,z_{N}(q_{1})=q_{1}
\]
is a $us$-path connecting $p_{1}$ to $q_{1}$, and it has $2N$ legs, each of them with length 
at most $L$. Hence, using the $us$-path connecting $p$ with $p_{1}$ and the $us$-path connecting $q$ with $q_{1}$, having at most
$l$ legs, each of them with length at most $R$, we 
have completed the proof.
\end{proof}
\begin{Corollary}\label{Corollary:Number and lengths of paths legs can be robustly limited}
Let $f: M\rightarrow M$ be a partially hyperbolic accessible diffeomorphism. Then there 
exist $l_{1}>0$ and $R_{1}>0$ such that for any $p, q \in M$ one can find a $us$-path 
$z_{0}=p, z_{1},  \dots , z_{l-1}, z_{l}=q, l\leq l_{1}$, that starts at $p$, ends 
at $q$, such that $q\in \mathcal{F}^{ss}_{R_{1}}(z_{l-1}, f)$, and each leg has length at most $R_{1}$.
\end{Corollary}

\begin{proof}
Let $l_{0}$ and $R_{0}$ be as in Lemma~\ref{Lemma:Number and lengths of paths legs can be limited}. Set 
$l_{1}=l_{0}+1$ and set $R_{1}=R_{0}$. Let $p, q \in M$. Take $q_{0}\in\mathcal{F}^{ss}_{R_{0}}(q,f)$. From 
Lemma~\ref{Lemma:Number and lengths of paths legs can be limited} we know that one can find a $us$-path 
$z_{0}=p, z_{1}, \dots , z_{l-1}=q_{0}$ that starts at $p$, ends at $q_{0}$, and has at most $l_{0}$ legs, each
of them with length at most $R_{0}=R_{1}$. Therefore,  
 $z_{0}=p, z_{1}, \dots , z_{l-1}=q_{0}, z_{l}=q$ is a $us$-path that starts at $p$, ends 
at $q$, and has at most $l_{1}$ legs, each of them with length at most $R_{1}$. 
\end{proof}

Now, using the last Corollary, we will prove that if $f$ is accessible then,
robustly, for a fixed $r>0$ and any pair of points $p,q\in M$
there exists a path connecting $p$ to the center unstable
disc of radius $r$ centred at $q$.

\begin{Lemma}\label{Lemma:Robust uniform continuity of foliations}
Let $f\in Diff^{r}(M)$ be a partially hyperbolic diffeomorphism. Given $R_{0}>0$ and $d_{0}>0$ 
there exist $\delta_{0}>0$ and a neighbourhood $\mathcal{U}(f)$ such that for any $g\in\mathcal{U}(f)$ and 
for every $x, y \in M$ such that $d(x,y)<\delta_{0}$ we have that $\mathcal{F}^{\alpha}_{R_{0}}(x,g)$ and 
$\mathcal{F}^{\alpha}_{R_{0}}(y,f)$ are $d_{0}$-close, $\alpha =ss$ or $uu$.
\end{Lemma}
\begin{proof}
From Stable Manifold Theorem we know that for every $x\in M$ there exist $r_{x}>0$ and a neighbourhood 
$\mathcal{U}_{x}(f)$ such that for any $g\in\mathcal{U}_{x}(f)$ and for every $y\in M$ such that 
$d(x,y)<r_{x}$ we have that $\mathcal{F}^{\alpha}_{R_{0}}(x,f)$ and $\mathcal{F}^{\alpha}_{R_{0}}(y,g)$ are 
$\frac{d_{0}}{2}$-close, $\alpha =ss$ or $uu$. Thus, for any $g\in\mathcal{U}_{x}(f)$ and for every 
$y, z\in B(x,r_{x})$ we get $\mathcal{F}^{\alpha}_{R_{0}}(y,g)$ and $\mathcal{F}^{\alpha}_{R_{0}}(z,f)$ are 
$d_{0}$-close, $\alpha =ss$ or $uu$. Since $M$ is compact, there are $x_{1},x_{2},\dots, x_{n}\in M$ such that 
\[
M\subset\bigcup^{n}_{i=1}B(x_{i},r_{x_{i}}).
\]
Let $\delta_{0}>0$ be a Lebesgue number of this cover and take 
\[
\mathcal{U}(f)=\bigcap^{n}_{i=1}\mathcal{U}_{x_{i}}(f).
\]
Thus, if $d(x,y)<\delta_{0}$ then $x,y\in B(x_{i}, r_{x_{i}})$ for some $i=1,2,\dots ,n$. Hence we have that 
$\mathcal{F}^{\alpha}_{R_{0}}(x,f)$ and $\mathcal{F}^{\alpha}_{R_{0}}(y,g)$ are $d_{0}$-close, $\alpha =ss$ 
or $uu$, for any $g\in\mathcal{U}(f)\subset\mathcal{U}_{x_{i}}(f)$. 
\end{proof}

\begin{Lemma}\label{Lemma:Robust transversal intersection}
Let $f\in Diff^{r}(M)$ be a partially hyperbolic diffeomorphism. Given $R_{0}>0$ and $r>0$ there exist 
$\epsilon>0$, $\delta_{0}>0$ and a neighbourhood $\mathcal{V}(f)$ such that for any $g\in\mathcal{V}(f)$ 
it follows that for any $x,y\in M$ with $d(x,y)<\delta_{0}$ the following holds:
\[
\mathcal{F}^{ss}_{R_{0}+\epsilon}(x,g)\pitchfork W^{cu}_{r}(z,g)\neq \emptyset\hspace{0.5 cm}\text{for any}\hspace{0.5 cm} z\in\mathcal{F}^{ss}_{R_{0}}(y,f).
\]
\end{Lemma}

\begin{proof}
Take $\epsilon >0$ given by Stable Manifold Theorem. There exists a neighbourhood $\mathcal{U}_{1}(f)$ such 
that $\epsilon >0$ can to be taken for any $g\in\mathcal{U}_{1}(f)$.
Let $d_{0}>0$ be such that for any $g\in\mathcal{U}_{1}(f)$ it follows that if $d(x,y)<d_{0}$ then 
\[
\mathcal{F}^{ss}_{\epsilon}(x,g)\pitchfork W^{cu}_{r}(y,g)\neq\emptyset.
\]
Consider $\mathcal{V}(f)\subset\mathcal{U}_{1}(f)$ and $\delta_{0}>0$ given by 
Lemma~\ref{Lemma:Robust uniform continuity of foliations}. Thus, if $g\in\mathcal{V}(f)$ and $x,y\in M$ 
with $d(x,y)< \delta_{0}$ we have that $\mathcal{F}^{ss}_{R_{0}}(x,g)$ and $\mathcal{F}^{ss}_{R_{0}}(y,f)$ 
are $d_{0}$-close and therefore 
\[
\mathcal{F}^{ss}_{R_{0}+\epsilon}(x,g)\pitchfork W^{cu}_{r}(z,g)\neq \emptyset\hspace{0.5 cm}\text{for any}\hspace{0.5 cm} z\in\mathcal{F}^{ss}_{R_{0}}(y,f).
\]

\end{proof}

\noindent Now we give an easy but interesting consequence of the last two Lemmas,
useful for our purposes in Section~\ref{sec:Property SH and Accessibility}.

\begin{Proposition}\label{Proposition:Robustly connecting to the center unstable}
Let $f\in Diff^{r}(M)$ be a partially hyperbolic accessible diffeomorphism. Given $r>0$ 
there exist a neighbourhood $\mathcal{U}(f)$, $l>0$ and $R>0$ such that for any $g\in\mathcal{U}(f)$ it 
follows that for every $p,q\in M$ there exists $q'\in W^{cu}_{r}(q,g)$ such that one can find a $us$-path 
by $g$ that starts at $p$, ends at $q'$, and has at most $l$ legs, each of them with length at most $R$. 
\end{Proposition}

\begin{proof}
Let $l_{1}>0$ and $R_{1}>0$ be as in 
Corollary~\ref{Corollary:Number and lengths of paths legs can be robustly limited}. For the sake 
of simplicity, we will assume that $l_{1}=4$. Given $R_{1}$ and $r$ let $\epsilon$, $\delta_{0}$ and 
$\mathcal{V}(f)$ be as in Lemma~\ref{Lemma:Robust transversal intersection}. From 
Lemma~\ref{Lemma:Robust uniform continuity of foliations} there exist $\delta_{1}>0$ and 
$\mathcal{U}_{1}(f)\subset\mathcal{V}(f)$ such that if $g\in\mathcal{U}_{1}(f)$ and $x,y\in M$ 
with $d(x,y)<\delta_{1}$ then $\mathcal{F}^{\alpha}_{R_{1}}(x,g)$ and $\mathcal{F}^{\alpha}_{R_{1}}(y,f)$ 
are $\delta_{0}$-close, $\alpha =ss$ or $uu$. Once again, using 
Lemma~\ref{Lemma:Robust uniform continuity of foliations} take $\delta_{2}>0$ and 
$\mathcal{U}_{2}(f)\subset\mathcal{U}_{1}(f)$ such that if $g\in\mathcal{U}_{2}(f)$ and $x,y\in M$ 
with $d(x,y)<\delta_{2}$ then $\mathcal{F}^{\alpha}_{R_{1}}(x,g)$ and $\mathcal{F}^{\alpha}_{R_{1}}(y,f)$ 
are $\delta_{1}$-close, $\alpha =ss$ or $uu$. Finally let $\mathcal{U}(f)\subset\mathcal{U}_{2}(f)$ be a 
neighbourhood such that for any $g\in\mathcal{U}(f)$ and for any $x\in M$ we have that 
$\mathcal{F}^{\alpha}_{R_{1}}(x,g)$ and $\mathcal{F}^{\alpha}_{R_{1}}(x,f)$ are $\delta_{2}$-close, 
$\alpha =ss$ or $uu$.\\
Let us prove that $\mathcal{U}(f)$, $l=l_{1}$ and $R=R_{1}+\epsilon$ satisfy what we want. Let 
$g\in\mathcal{U}(f)$ and let $p,q\in M$. We know that there exists a $us$-path by $f$ that starts at
$p$, ends at $q$, and has at most $l_{1}$ legs, each of them with length at most $R_{1}$. Moreover, the 
last leg lies in $\mathcal{F}^{ss}_{R_{1}}(q,f)$. Suppose that such a $us$-path has exactly $l_{1}$ 
legs. Let $p=z_{0}$, $z_{1}$, $z_{2}$, $z_{3}$, $z_{4}=q$ be such a $us$-path.\\
We have that $\mathcal{F}^{uu}_{R_{1}}(p,g)$ and $\mathcal{F}^{uu}_{R_{1}}(p,f)$ are $\delta_{2}$-close.
Then, there exists $x_{1}\in\mathcal{F}^{uu}_{R_{1}}(p,g)$ such that $d(x_{1},z_{1})<\delta_{2}$. Thus 
$\mathcal{F}^{ss}_{R_{1}}(x_{1},g)$ and $\mathcal{F}^{ss}_{R_{1}}(z_{1},f)$ are $\delta_{1}$-close. Therefore, 
there exists $x_{2}\in\mathcal{F}^{ss}_{R_{1}}(x_{1},g)$ such that $d(x_{2},z_{2})<\delta_{1}$. Hence 
$\mathcal{F}^{uu}_{R_{1}}(x_{2},g)$ and $\mathcal{F}^{uu}_{R_{1}}(z_{2},f)$ are $\delta_{0}$-close. Take 
$x_{3}\in\mathcal{F}^{uu}_{R_{1}}(x_{2},g)$ with $d(x_{3},z_{3})<\delta_{0}$. From 
Lemma~\ref{Lemma:Robust transversal intersection}, since $q\in\mathcal{F}^{ss}_{R_{1}}(z_{3},f)$ we have that
\[
\mathcal{F}^{ss}_{R_{1}+\epsilon}(x_{3},g)\pitchfork W^{cu}_{r}(q,g)\neq \emptyset.
\]
Take 
\[
q'\in\mathcal{F}^{ss}_{R_{1}+\epsilon}(x_{3},g)\pitchfork W^{cu}_{r}(q,g),
\]
and $p,x_{1},x_{2},x_{3},q'$ the $us$-path by $g$. The case that such a $us$-path by $f$, connecting 
$p$ to $q$, has $l'$ legs with $l'<l_{1}$, is similar.

\end{proof}

\subsection{Property SH}
\label{subsec:Property SH} \quad
\newline

\noindent We define below the key property introduced in \cite{PS} which ensures the robustness of the minimal stable foliation.
Moreover, we will prove later that this property also ensures robust transitivity. Before we do, let 
us introduce some notation: if $L:V\rightarrow W$ is a linear isomorphism between normed vector 
spaces we denote by $m\lbrace L\rbrace$ the minimum norm of $L$, i.e. 
$m\lbrace L\rbrace=\parallel L^{-1}\parallel^{-1}$. 
\begin{Definition}\label{Definition:The SH property}
Let $f\in Diff^{r}(M)$ be a partially hyperbolic diffeomorphism. We say that $f$ exhibits 
the property SH  if there exist $\lambda_{0}>1, C>0$ such that for any $x\in M$ there 
exists $y^{u}(x)\in\mathcal{F}^{uu}_{1}(x,f)$ (the ball of radius $1$ in $\mathcal{F}^{uu}(x,f)$ 
centred at $x$) satisfying
\[
m\lbrace Df^{n}_{\mid E^{c}(f^{l}(y^{u}(x)))}\rbrace> C\lambda^{n}_{0}\hspace{0.5 cm}\text{for any}\hspace{0.5 cm} n>0,\hspace{0.5 cm} l>0.
\]
\end{Definition}
\noindent The Property SH persists under slight perturbations. 
\begin{Theorem}\label{Theorem:SH is robust}
Let $f\in Diff^{r}(M)$ be a partially hyperbolic diffeomorphism exhi\-
biting Property SH. Then, 
there exist a $C^{1}$ neighbourhood $\mathcal{U}$ of $f$, $C'>0$ and $\sigma>1$ such that for any $g\in\mathcal{U}$ it follows that 
for any $x\in M$ there exists $y^{u}\in\mathcal{F}^{uu}_{1}(x,g)$ satisfying 
\[
m\lbrace Dg^{n}_{\mid E^{c}(g^{l}(y^{u}))}\rbrace> C' \sigma^{n} \hspace{0.5 cm}\text{for any}\hspace{0.5 cm} n>0,\hspace{0.5 cm} l>0.
\]
\end{Theorem}
\begin{proof}
See \cite{PS}.
\end{proof}

Before stating the Theorem which guarantees the robustness of the minimality
of a strong stable foliation for a partially hyperbolic diffeomorphism, we recall the concept of minimal stable foliation.

\begin{Definition}\label{Definition:Minimal and robustly minimal stable foliation }
Let $f:M\rightarrow M$ be a $C^{r}$ partially hyperbolic diffeomorphism.
We say that $\mathcal{F}^{ss}(f)$ is minimal  when $\mathcal{F}^{ss}(x,f)$,
the leave of this foliation passing through the point $x$, is dense in $M$
for every $x\in M$.  
We say that $\mathcal{F}^{ss}(f)$ is $C^{r}$-robustly minimal
if there exist a $C^{r}$ neighbourhood $\mathcal{U}(f)$ such that
$\mathcal{F}^{ss}(g)$ is minimal for every diffeomorphism $g\in\mathcal{U}(f)$.
\end{Definition}

\begin{Theorem}\label{Theorem:Stable foliation is robustly minimal}
Let $r\geq 1$ and let $f\in Diff^{r}(M)$ be a partially hyperbolic diffeomorphism
satisfying Property SH and such that the strong stable foliation 
$\mathcal{F}^{ss}(f)$ is minimal. Then, $\mathcal{F}^{ss}(f)$ is $C^{1}$
(and hence $C^{r}$) robustly minimal. 
\end{Theorem}

\begin{proof}
See \cite{PS}.
\end{proof}

\subsection{Blenders and Heterodimensional Cycles}
\label{subsec:Blenders and Heterodimensional Cycles} \quad
\newline

\noindent In this subsection we recall the notions of blender and
heterodimensional cycle and the relation between them. We
also give a condition under which the presence of a blender
guarantees the Property SH.

\begin{Definition}\label{Definition:Blender}
A cs\emph{-blender} for $f\in Diff^{r}(M)$ with $r\geq 1$  is a hyperbolic set $K$ with a partially hyperbolic structure
$E^{ss}\bigoplus E^{c}\bigoplus E^{uu}$ such that $E^{c}\bigoplus E^{uu}$ is the unstable bundle, and with a periodic
point $p$ such that for any disc $D$ that is $C^{1}$-close to $\mathcal{F}_{loc}^{uu}(p)$, there exists $x$ in the
hyperbolic set $K$ such that  $\mathcal{F}^{ss}(x)$ intersects $D$. Moreover, such a property is $C^{1}$-persistent.\\
A cs\emph{-blender} for $f^{-1}$ is called cu\emph{-blender} for $f$. 
\end{Definition}

\begin{Proposition}\label{Proposition:Enrique}
Let $f\in Diff^{r}(M)$ be a partially hyperbolic diffeomorphism with strong
unstable minimal foliation such that $K$ is a $cs$-blender for $f$ with a periodic point $p$. Then
$f$ satisfies Property SH.\\
Analogously if $f$ has a strong stable minimal foliation and it has
a $cu$-blender then $f^{-1}$ satisfies Property SH.
\end{Proposition}
\begin{proof}
It is not difficult to see that if the strong unstable foliation is minimal,
then there exists $r > 0$ such that 
$D_{z}\subset\mathcal{F}_{r}^{uu}(z,f),\forall z\in M$,
where $D_{z}$ is a disc $C^{1}$-close to $\mathcal{F}_{loc}^{uu}(p,f)$.
Hence, using the Definition~\ref{Definition:Blender} above, we have that for some $l > 0$
and for every $z\in M$, there exists $y^{z}$ such that:
\begin{equation*}
y^{z}\in\mathcal{F}^{uu}_{r}(z,f)\cap \mathcal{F}^{ss}_{l}(x,f)\quad \text{for some}\quad x\in K. 
\end{equation*}
From this, it follows that
\begin{equation*}
d(f^{n}(y^{z}), f^{n}(x))\rightarrow 0\quad\text{when}\quad n\rightarrow +\infty. 
\end{equation*}
Since $T_{x}M=E^{ss}(x,f)\oplus E^{u}(x,f)\oplus E^{uu}(x,f)$ for
every $x\in K$, 
$Df_{\mid E^{c}(x)}$ is uniformly expanding in the future.
Therefore, $f$ satisfies Property SH.
\end{proof}

\noindent Blenders can be produced by unfolding heterodimensional cycles 
far from homoclinic tangencies as in next  Proposition found in \cite{BDV}. 

\begin{Definition}\label{Definition:Heterodimensional cycles}
Given a diffeomorphism $f$ with two hyperbolic periodic points
$P_{f}$ and $Q_{f}$ with different indices, say 
$\text{index}(P_{f})>\text{index}(Q_{f})$, we say that $f$ has a
heterodimensional cycle with codimension $\text{index}(P_{f})-\text{index}(Q_{f})$
associated to $P_{f}$ and $Q_{f}$ if $\mathcal{F}^{s}(P_{f},f)$ and
$\mathcal{F}^{u}(Q_{f},f)$ have a (non-trivial) transverse intersection
and $\mathcal{F}^{u}(P_{f},f)$ and
$\mathcal{F}^{s}(Q_{f},f)$ have a quasi-transverse intersection along
the orbit of some point $x$, i.e., $T_{x}\mathcal{F}^{u}(P_{f},f)+T_{x}\mathcal{F}^{s}(Q_{f},f)$
is a direct sum.
\end{Definition}

\begin{Proposition}\label{Proposition:Bonatti, Diaz, Viana}
Let $f$ be a $C^{1}$ diffeomorphism with a heterodimensional cycle
associated to saddles $P$ and $Q$ of indices $p$ and $q=p+1$.
Suppose that the cycle is $C^{1}$-far from homoclinic tangencies.
Then there is an open set $\mathcal{V}\subset Diff^{1}(M)$ whose
closure contains $f$ such that for every $g$ in  $\mathcal{V}$
there are a $cs$-blender defined for $g$ and a $cs$-blender
defined for $g^{-1}$ such that:

$\cdotp$ The $cs$-blender for $g$ is associated to a hyperbolic periodic point $R_{g}$.

$\cdotp$ The $cs$-blender for $g^{-1}$ is associated to a hyperbolic periodic point $S_{g}$.
 
\end{Proposition}

%% file: PropertySHandTopologicallyMixing.tex
\section{Property SH and Topologically Mixing}
\label{sec:Property SH and Topologically Mixing}

\noindent Our first goal is to show that some diffeomorphisms with Property SH are
topolo\-gically mixing. In order to do this
we will need a few preliminary results.

\begin{Lemma}\label{Lemma:From the Stable Manifold Theorem}
Let $\epsilon >0$ be given by the Stable Manifold Theorem and $r>0$ sufficiently small.
For any $\epsilon '<\epsilon, r'<r$ there exists $d'=d'(\epsilon ',r')>0$ such that
for any pair of points $x,y\in M$ with $dist(x,y)<d'$ the manifolds $W^{cu}_{r'}(x,f)$
and $\mathcal{F}^{ss}_{\epsilon'}(y,f)$ intersect transversally in
exactly one point. 
\end{Lemma}

\begin{Lemma}\label{Lemma:Encapsulating cylinder}
Given $L>0$ and $r_{0}>0$ there exist $d,r_{1}$ and $\epsilon_{1}$ with $0< d< r_{1}<r_{0}$, $\epsilon_{1}>0$ and
such that for every $x\in M, z\in W^{cu}_{d}(x,f)$ if
\begin{equation*}
A_{x,z}=W^{cu}_{r_{1}}(x,f)\cup(\bigcup_{y\in W^{cu}_{d}(z,f)}\mathcal{F}^{ss}_{\epsilon_{1}}(y,f))
 \end{equation*}
then $diam (A_{x,z})< L$ and for any $y\in W^{cu}_{d}(z,f)$ holds that
$\mathcal{F}^{ss}_{\epsilon_{1}}(y,f)$ intersect $W^{cu}_{r_{1}}(x,f)$
transversally at exactly one point.
\end{Lemma}
\begin{proof}
Let $\epsilon >0$, $r>0$ be as in Lemma~\ref{Lemma:From the Stable Manifold Theorem}.
Take $r_{1}< min\{\frac{L}{8}, r, r_{0}\}$ and $\epsilon_{1}< min \{\frac{L}{8},\epsilon\}$. From 
Lemma~\ref{Lemma:From the Stable Manifold Theorem}, there exist $d_{1}$ such that if
$dist (x,y)< d_{1}$ then the manifolds $W^{cu}_{r_{1}}(x,f)$ and $\mathcal{F}^{ss}_{\epsilon_{1}}(y,f)$
intersect transversally at exactly one point.

\noindent Choose $d< min\{\frac{d_{1}}{4}, r_{1}\}$.
Now let $x\in M$ be arbitrary  and $z\in W^{cu}_{d}(x,f)$.
Observe that if $y\in W^{cu}_{d}(z,f)$ then

\begin{equation*}
dist(x,y)\leq dist(x,z)+dist(z,y)\leq dist_{W_{d}^{cu}(x,f)}(x,z)+dist_{W_{d}^{cu}(z,f)}(z,y)\leq 2d< d_{1},
\end{equation*}
so the manifolds $W^{cu}_{r_{1}}(x,f)$ and $\mathcal{F}^{ss}_{\epsilon_{1}}(y,f)$
intersect transversally at exactly one point.\\
\noindent Also, for some $y\in W^{cu}_{d}(z,f)$, if 
$t\in\mathcal{F}^{ss}_{\epsilon_{1}}(y,f)$  then
\begin{equation*}
dist(t,x)\leq dist(t,y)+dist (y,x)\leq dist_{\mathcal{F}_{\epsilon_{1}}^{ss}(y,f)}(t,y)+dist(y,x)\leq \epsilon_{1}+2d<\frac{L}{2} 
\end{equation*}
and if $\overline{t}\in W^{cu}_{r_{1}}(x,f)$ then 

\begin{equation*}
dist(\overline{t},x)\leq dist_{W_{r_{1}}^{cu}(x,f)}(\overline{t},x)\leq r_{1}<\frac{L}{8}.
\end{equation*} 

\noindent Hence $diam(A)< L$.
\end{proof}

\begin{Proposition}\label{Proposition:Propiedad 1}
If $f:M\rightarrow M$ is a partially hyperbolic diffeomorphism and satisfies Property SH, 
then for any center-unstable disc $D$, there exists a periodic hyperbolic point $p$ of stable dimension $dim(E^{ss})$, whose stable manifold meets $D$ transversally.
\end{Proposition}

\begin{proof}
 
 Let $x$ be a point in $M$ and $D=W_{\beta}^{cu}(x,f)$ a center-unstable disc. For  $L=\frac{\beta}{2}$, 
choose $d$, $r_{1}$, $\epsilon_{1}$ and $A_{x,z}$ given by Lemma~\ref{Lemma:Encapsulating cylinder}.
Observe that if $t\in W_{r_{1}}^{cu}(x,f)$
then $d_{W_{r_{1}}^{cu}(x,f)}(t,x)\leq diam(A_{x,z})< L<\beta$. So

\begin{equation*}
\mathcal{F}_{d}^{uu}(x,f)\subset W_{d}^{cu}(x,f)\subset W_{r_{1}}^{cu}(x,f)\subset W_{\beta}^{cu}(x,f)
\end{equation*}

 Take $n_{0}$ such that $\mathcal{F}^{uu}_{1}(f^{n_{0}}(x),f)\subset f^{n_{0}}(\mathcal{F}^{uu}_{d}(x,f))$. Consider 
the point $y^{u}\in\mathcal{F}^{uu}_{1}(f^{n_{0}}(x),f)$ satisfying

\begin{equation}\label{Primera}
m\lbrace Df^{n}_{\mid E^{c}(f^{l}(y^{u}))}\rbrace> C\sigma^{n}\hspace{0.5 cm}\text{for any}\hspace{0.5 cm} n>0,\hspace{0.5 cm} l>0,
\end{equation}

\noindent where $C>0,\sigma >1$.\\
\noindent We may assume that $C=1$. Otherwise we take a fixed power of $f$. Make
$\lambda = \sigma^{-1}$, fix $\lambda_{1}\in (\lambda,1)$ and take $r_{0}$ as in 
Lemma~\ref{Lemma:Backward iterations contract}. Let $\eta >0$ be such that 

\begin{equation}\label{Segunda}
f^{-n_{0}}(W_{\eta}^{cu}(y^{u},f))\subset W_{d}^{cu}(f^{-n_{0}}(y^{u}),f).
\end{equation}

\noindent Choose $q\in \omega(y^{u})$, $\omega$-limit of $y^{u}$, being a recurrent point.
For $\epsilon >0$, there exists $\delta >0$ such that

\[
d(z_{1},z_{2})<\delta\Rightarrow\mathcal{F}^{ss}_{\epsilon}(z_{1},f)\pitchfork W_{r_{0}}^{cu}(z_{2},f)\neq\emptyset.
\]

From the Shadowing Lemma, there exists a periodic hyperbolic point $p\in M$, shadowing a periodic 
pseudo-orbit in $\omega(y^{u})$, constructed by means of the recurrent point $q$, with 
$d(p,q)<\frac{\delta}{2}$. Since $q\in \omega(y^{u})$, take $m\in\mathbb{N}^{*}$ such that 
$\lambda_{1}^{m}r_{0} <\eta$ and $d(f^{m}(y^{u}),q) < \frac{\delta}{2}$.
Now set $k_{0}=n_{0}+m$. Then $d(p,f^{m}(y^{u}))<\delta$ and therefore

\begin{equation}\label{Tres}
\mathcal{F}^{ss}_{\epsilon}(p,f)\pitchfork W_{r_{0}}^{cu}(f^{m}(y^{u}),f)\neq\emptyset.
\end{equation}

We know that $E^{cu}=E^{c}\oplus E^{u}$ is a dominated decomposition. Thus, there is 
$K>0$ such that 
$\parallel Df^{-n}_{\mid E^{cu}}\parallel\leq K \sup\lbrace\parallel Df^{-n}_{\mid E^{c}}\parallel,\parallel Df^{-n}_{\mid E^{u}}\parallel\rbrace$. For the sake of simplicity, we will assume that $K=1$. From 
\eqref{Primera} we have that
\begin{equation*}
\prod_{j=0}^{n}\parallel Df^{-1}_{\mid E^{c}(f^{-j}(f^{m}(y^{u})))}\parallel < \lambda^{n}, \quad 0\leq n\leq m
\end{equation*}
and therefore
\begin{equation*}
\prod_{j=0}^{n}\parallel Df^{-1}_{\mid E^{cu}(f^{-j}(f^{m}(y^{u})))}\parallel < \lambda^{n}, \quad 0\leq n\leq m.
\end{equation*}
From Lemma~\ref{Lemma:Backward iterations contract} we conclude 
that $f^{-m}(W_{r_{0}}^{cu}(f^{m}(y^{u}),f))\subset W_{\lambda_{1}^{m}r_{0}}^{cu}(y^{u},f)\subset W_{\eta}^{cu}(y^{u},f)$ 
and hence, using \eqref{Segunda}, we have 

\[
f^{-k_{0}}(W_{r_{0}}^{cu}(f^{m}(y^{u}),f))\subset f^{-n_{0}}(W_{\eta}^{cu}(y^{u},f))\subset W_{d}^{cu}(f^{-n_{0}}(y^{u}),f).
\] 

\noindent So, from \eqref{Tres}, 
\[f^{-k_{0}}(\mathcal{F}_{\epsilon}^{ss}(p,f))\pitchfork W_{d}^{cu}(z,f)\neq\emptyset
\]

\noindent where $z=f^{-n_{0}}(y^{u})\in\mathcal{F}_{d}^{uu}(x,f)\subset W_{d}^{cu}(x,f)$.

\noindent Now take $y\in f^{-k_{0}}(\mathcal{F}_{\epsilon}^{ss}(p,f))\cap W_{d}^{cu}(z,f)$. As $z\in W_{d}^{cu}(x,f)$, 
from Lemma~\ref{Lemma:Encapsulating cylinder} we conclude that $\mathcal{F}_{\epsilon_{1}}^{ss}(y,f)$ intersects transversally
$W_{r_{1}}^{cu}(x,f)$ in exactly one point and therefore $\mathcal{F}^{ss}(f^{-k_{0}}(p))$
intersects transversally $W_{r_{1}}^{cu}(x,f)\subset D$.

\end{proof}

\begin{Proposition}\label{Proposition:Propiedad 2}
Let $f:M\rightarrow M$ be a partially hyperbolic and transitive diffeomorphism. For any open set $U$ and any center-unstable disc $D$, there exists a local strong stable disc $\mathcal{F}_{\epsilon}^{ss}(x,f)$ contained in $U$ with a negative iterated which intersects $D$ transversally.
\end{Proposition}

\begin{proof}
Just take $x\in U$ whose orbit is dense in $M$ and $\epsilon >0$ such that $\mathcal{F}_{\epsilon}^{ss}(x,f)\subset U$. 
There exists a sequence of negative iterates of $x$, $(f^{n_{k}}(x))_{k}$, converging to the center of the disc $D$. Taking $k$
big enough we can guarantee that $f^{n_{k}}(\mathcal{F}_{\epsilon}^{ss}(x,f))$ intersects $D$ transversally.

\end{proof}

\begin{Remark}\label{Remark:Puntos periodicos uniformes}
Periodic hyperbolic points whose existence is proven in  Proposition~\ref{Proposition:Propiedad 1} can be taken arbitrarily close to the
$w$-limit of a point $z$ such that $m\lbrace Df^{n}_{\mid E^{c}(f^{l}(z))}\rbrace> C\lambda^{n}_{0}$ for any $n>0$, $l>0$ like in  Definition~\ref{Definition:The SH property}. Consequently these periodic
hyperbolic points are chosen uniformly expanding in the central direction.
 \end{Remark}

\begin{Theorem}\label{Theorem:Observacion 11 del Referee}
Let $f\in Diff^{r}(M)$ be a partially hyperbolic diffeomorphism with the Property $SH$ and such that
$f^{n}$ is transitive for each $n\geq 1$. Then $f$ is topologically mixing.
\end{Theorem}

\begin{proof}
Let $\mathcal{U},\mathcal{V}\subset M$ be open sets. Take $x\in\mathcal{U}$ arbitrary and $\eta>0$
such that $D=W_{\eta}^{cu}(x,f)\subset\mathcal{U}$. From Proposition~\ref{Proposition:Propiedad 1},
there exists a periodic hyperbolic point $p$ such that $\mathcal{F}^{ss}(p,f)$ intersects $D$.
Assume $k$ to be the period of $p$. Since $f^{k}$ is transitive, using 
Proposition~\ref{Proposition:Propiedad 2}, there exists a local strong stable disc 
$\mathcal{F}_{\epsilon}^{ss}(x,f)\subset\mathcal{W}$ with a negative iterated of $f^{k}$,
say $f^{-kl}$, which intersects $W_{r}^{cu}(p,f)$ for some $r$ sufficiently small.
Thus applying $\lambda$-Lemma for $f^{k}$, we get $n_{0}\in\mathbb{N}$ such that
\[
\mathcal{F}_{\epsilon}^{ss}(x,f)\pitchfork f^{nk}(D)\neq\varnothing,\quad\forall n\geq n_{0}.
\]

\noindent Hence, 

\[
f^{-kl}(W)\cap f^{nk}(\mathcal{U})\neq\varnothing,\quad\forall n\geq n_{0}.
\]
\noindent Therefore $f^{k}$ is topologically mixing. Consequently $f$ is topologically mixing.
\end{proof}

\noindent Now we give our version of Brin's Theorem. 
Observe that the condition of accessibility in relation to open sets is weaker than the
condition of accessibility in the original version of Brin's Theorem.\\

\noindent Let $g\in Diff^{r}(M)$. We will denote by $\Omega(g)$ the set of the non-wandering points for $g$.

\begin{Theorem}\label{Theorem:Brin's Theorem}
Let  $f\in Diff^{r}(M)$ be a partially hyperbolic 
diffeomorphism exhibiting the accessibility property in relation to open sets. If 
$\Omega(f)=M$ then $f$ is transitive.
\end{Theorem}
\begin{proof}
The proof is similar to the original version in \cite{P}.
\end{proof}

\noindent The condition of $f^{n}$ being transitive for each $n\geq 1$ is 
implied by the following Proposition:

\begin{Proposition}\label{Proposition:Observacion 14 del Referee}
Let $f\in Diff^{r}(M)$ be a partially hyperbolic diffeomorphism. If $\Omega (f)=M$ then $\Omega (f^{n})=M$
for each $n\geq 1$. In particular, $\Omega (f)=M$ and $f$ accessible in relation to open sets imply
that $f^{n}$ is transitive for each $n\geq 1$.
\end{Proposition}

\begin{proof}
Let $n \geq 1$ and let $U$ be an open set in $M$. As $\Omega (f)=M$ there exist $x_{0}\in U$ and $k_{0}\geq 1$
such that $f^{k_{0}}(x_{0})\in U$. By continuity, there exists an open set $U_{0}$ with $U_{0}\subset U$,
$x_{0}\in U_{0}$ and $f^{k_{0}}(U_{0})\subset U$. This way we can define recurrently a sequence of points
$x_{0},x_{1},\cdots\in M$, a sequence of positive natural numbers $k_{0},k_{1},\cdots$ and a sequence
$U_{0},U_{1},\cdots$ of open sets in $M$ such that for any $i\geq 0$ we have $x_{i+1}\in f^{k_{i}}(U_{i})$,
$f^{k_{i+1}}(x_{i+1})\in  f^{k_{i}}(U_{i})$, $x_{i+1}\in U_{i+1}\subset f^{k_{i}}(U_{i})$ and
$f^{k_{i+1}}(U_{i+1})\subset f^{k_{i}}(U_{i})$.

It is easy to see that for any pair of non negative integers $r,s$ holds that\\ 
$f^{-k_{r}-k_{r+1}-\cdots-k_{r+s}}(U_{r+s+1})\subset U_{r}$ and  $U_{r}\subset U$. Now define the numbers $t_{i}$
by $t_{i}=k_{0}+k_{1}+\cdots+k_{i}$. Let $j$ and $l\neq 0$ be such that $t_{j+l}\equiv t_{j}\quad mod (n)$.
Then there exist an integer $m$ such that $k_{j+1}+\cdots+k_{j+l}=mn$.
If we set $W=f^{-k_{j+1}-\cdots-k_{j+l}}(U_{j+l+1})$ we know that $W\subset U_{j+1}\subset U$
and that $f^{mn}(W)=U_{j+l+1}\subset U$. Hence taking $y\in W\subset U$ we have $f^{mn}(y)\in U$. 

\end{proof}

\begin{Remark}
Assume that a partially hyperbolic diffeomorphism $f$ has the accessibility property. If $f$ is 
transitive then $f^{n}$ is also transitive for every $n\in\mathbb{Z}^{*}$.
\end{Remark}

\begin{Corollary}\label{Corollary:Observacion 17 del Referee}
Let $f\in Diff^{r}(M)$ be a partially hyperbolic diffeomorphism, accessible in relation to open sets,
satisfying $\Omega (f)=M$ and the Property SH, then $f$ is topologically mixing.
\end{Corollary}

\begin{Corollary}
Let $f$ be a partially hyperbolic diffeomorphism, accessible, topolo\-gically transitive and satisfying
Property SH. Then $f$ is topologically mixing.
\end{Corollary}

%% file: PropertySHandAccessibility.tex
\section{Property SH and Accessibility}
\label{sec:Property SH and Accessibility}

\noindent In this section, we follow with other results, which provide facts about 
the accessibility classes, accessibility in relation to open sets and robust transitivity, considering Property SH.

\begin{Theorem}\label{Theorem:Robustness of density of accessibility classes}
Let $f\in Diff^{r}(M)$ 
be a partially hyperbolic accessible diffeomorphism exhibiting Property SH. Then 
there exists a $C^{1}$ neighbourhood  of $f$, $\mathcal{U}=\mathcal{U}(f)$, such that for every $g\in\mathcal{U}$ and $p\in M$ 
it follows that $\mathcal{C}(p,g)$ is dense in $M$.
\end{Theorem}

\begin{proof}
From Theorem~\ref{Theorem:SH is robust} we know that there exist a neighbourhood 
$\mathcal{U}_{0}(f)$, $C'> 0$ and $\sigma> 1$ such that for every $g\in\mathcal{U}_{0}$ and $x\in M$ 
there exists a point $y^{u}\in\mathcal{F}^{uu}_{1}(x,g)$ satisfying

\begin{equation}\label{A}
m\lbrace Dg^{n}_{\mid E^{c}(g^{l}(y^{u}))}\rbrace> C'\sigma^{n}\hspace{0.5 cm}\text{for any}\hspace{0.5 cm} n>0,\hspace{0.5 cm} l>0.
\end{equation}

We may assume that $C=1$. Otherwise we take a fixed power of every $g\in\mathcal{U}_{0}$. Let 
$\lambda = \sigma^{-1}$ and fix $0 <\lambda <\lambda_{1} < 1$ and let $r$ be as in 
Lemma~\ref{Lemma:Backward iterations contract}. For this $r > 0$ take 
$\mathcal{U}(f)\subset\mathcal{U}_{0}(f)$, $l > 0$ and $R > 0$ as in 
Proposition~\ref{Proposition:Robustly connecting to the center unstable}. We will prove that for every 
$g\in\mathcal{U}(f)$ and $p\in M$ we have that $\mathcal{C}(p,g)$ is dense in $M$.\\
Let $\mathcal{V}\subset M$ be an open set and let
$z\in\mathcal{V}$. Let $\beta > 0$ be such that 
$\mathcal{F}^{uu}_{\beta}(z,g)\subset\mathcal{V}$. Take $n_{0}$ such that 
$g^{n_{0}}(\mathcal{F}^{uu}_{\beta}(z,g))\supset\mathcal{F}^{uu}_{1}(g^{n_{0}}(z),g)$. Consider 
the point $y^{u}\in\mathcal{F}^{uu}_{1}(g^{n_{0}}(z),g)$ given by Theorem~\ref{Theorem:SH is robust} and 
let $\eta > 0$ be such that 

\begin{equation}\label{B}
g^{-n_{0}}(W^{cu}_{\eta}(y^{u},g))\subset\mathcal{V}.
\end{equation}
Choose a positive integer $m$ such that $\lambda^{m}_{1} r <\eta$ and set $k=n_{0}+m$. From 
Proposition~\ref{Proposition:Robustly connecting to the center unstable} for $q = g^{m}(y^{u})$ there 
exists $q'\in W^{cu}_{r}(q,g)$ such that one can find a $us$-path by $g$ that starts at $g^{k}(p)$, ends 
at $q'$, and has at most $l$ legs, each of them with length at most $R$.\\
Since $E^{cu}=E^{c}\oplus E^{u}$ and this decomposition is dominated, there is $L > 0$ such that 
$\parallel Dg^{-n}_{\mid E^{cu}}\parallel\leq L \sup\lbrace\parallel Dg^{-n}_{\mid E^{u}}\parallel ,\parallel Dg^{-n}_{\mid E^{c}}\parallel\rbrace$.
For the sake of simplicity, we will assume that $L = 1$. From \eqref{A} we know that 
\begin{equation*}
\prod_{j=0}^{n}\parallel Dg^{-1}_{\mid E^{c}(g^{-j+m}(y^{u}))}\parallel < \lambda^{n}, \quad 0\leq n\leq m
\end{equation*}
and therefore
\begin{equation*}
\prod_{j=0}^{n}\parallel Dg^{-1}_{\mid E^{cu}(g^{-j+m}(y^{u}))}\parallel < \lambda^{n}, \quad 0\leq n\leq m.
\end{equation*}
From Lemma~\ref{Lemma:Backward iterations contract} we conclude that
\begin{equation*}
g^{-m}(W_{r}^{cu}(g^{m}(y^{u}),g))\subset W_{\lambda_{1}^{m}r}^{cu}(y^{u},g)\subset W_{\eta}^{cu}(y^{u},g)
\end{equation*}
and hence, using \eqref{B}, we have $g^{-k}(W_{r}^{cu}(g^{m}(y^{u}),g))\subset \mathcal{V}$. Since 
$q'\in W_{r}^{cu}(g^{m}(y^{u}),g)$ we get $g^{-k}(q')\in\mathcal{V}$. Thus, there 
exists a $us$-path by $g$ that starts at $p$, ends at $g^{-k}(q')\in\mathcal{V}$. Hence,
$g^{-k}(q')\in\mathcal{V}\cap\mathcal{C}(p,g)$ and the proof is completed.
\end{proof}

\begin{Corollary}\label{Corollary:Robustness of accesibility by open sets}
Let $f\in Diff^{r}(M)$ 
be a partially hyperbolic accessible diffeomorphism exhibiting Property SH. Then 
there exists a $C^{1}$ neighbourhood of $f$, $\mathcal{U}=\mathcal{U}(f)$, such that for 
any $g\in\mathcal{U}$ it follows that $g$ is accessible in  relation to open sets.
\end{Corollary}

\begin{Corollary}\label{Corollary:Generalization of the first result}
Let $f\in Diff^{r}(M)$  be a partially hyperbolic accessible diffeomorphism 
exhibiting Property SH and such that $\Omega(f)=M$. Then any diffeomorphism $g$ being
$C^{1}$-close to $f$  and such that $\Omega(g)=M$ is topologically mixing.
\end{Corollary}

\begin{Corollary}\label{Corollary:First result}
Let $f\in Diff^{r}(M)$  be a partially hyperbolic, accessible, volume preserving diffeomorphism
exhibiting the Property SH, then any diffeomorphism $C^{1}$-close to $f$ that is volume preserving
is topologically mixing.
\end{Corollary}

%% file: PropertySHandRobustTransitivity.tex
\section{Property SH and Robust Transitivity}
\label{sec:Property SH and Robust Transitivity}
\noindent Unlike the results in preceding section our next Theorem 
do not have in the hypotheses the condition of Accessibility.
Property SH is enough  to guarantee robust transitivity.

\begin{Lemma}\label{Lemma:Consequence of the Stable Manifold Theorem}
Let $f\in Diff^{r}(M)$ be a partially hyperbolic diffeomorphism. There exist 
$\epsilon > 0$ such that given $r > 0$ there are $\delta> 0$ and a neighbourhood $\mathcal{V}_{0}$
of $f$ such that for any $x,y\in M$ with $d(x,y)< \delta$ it follows that
\begin{align*}
\bullet&\hspace{0.5 cm}\mathcal{F}^{ss}_{\epsilon}(x,g)\pitchfork\mathcal{W}^{cu}_{r}(y,g)\neq\emptyset\\
\bullet\bullet&\hspace{0.5 cm}\mathcal{F}^{uu}_{\epsilon}(x,g)\pitchfork\mathcal{W}^{cs}_{r}(y,g)\neq\emptyset,\quad\text{for any}\quad g\in\mathcal{V}_{0}.
\end{align*}
  
\end{Lemma}
\begin{proof}
 The result follows from Stable Manifold Theorem.
\end{proof}

\begin{Theorem}\label{Theorem:Third result}
Let $M$ be a compact Riemannian manifold and let  $f\in Diff^{r}(M)$ be a partially hyperbolic diffeomorphism,
non-hyperbolic, transitive. If $f$ and $f^{-1}$ satisfy Property SH then $f$ is robustly transitive.
\end{Theorem}
\begin{proof}
For the sake of clarity we divide the proof in two steps. The first step
deals with the construction of an appropiate neighborhood $\mathcal{V}$ of $f$. 
In the second step we prove that any diffeomorphism in $\mathcal{V}$ is transitive.\\
\textbf{Step 1}\\
From Theorem~\ref{Theorem:SH is robust} there exist a neighborhood  $\mathcal{V}_{1}(f)$, $C_{0}>0$ and 
$\sigma_{0}>1$ such that for every $g\in\mathcal{V}_{1}$ and $x\in M$ there exists a point 
$y\in\mathcal{F}^{uu}_{1}(x,g)$ such that 
\begin{equation}\label{Cuarta}
m\lbrace Dg^{n}_{\mid E^{c}(g^{l}(y))}\rbrace> C_{0}\sigma^{n}_{0}\hspace{0.5 cm}\text{for any}\hspace{0.5 cm} n>0,\hspace{0.5 cm} l>0.
\end{equation}
Analogously there exist a neighborhood $\mathcal{V}_{2}(f^{-1})$, $C_{1}>0$ and 
$\sigma_{1}>1$ such that for every $h\in\mathcal{V}_{2}$ and $x\in M$ there exists a point 
$y\in\mathcal{F}^{uu}_{1}(x,h)$ such that
\[
m\lbrace Dh^{n}_{\mid E^{c}(h^{l}(y))}\rbrace> C_{1}\sigma^{n}_{1}\hspace{0.5 cm}\text{for any}\hspace{0.5 cm} n>0,\hspace{0.5 cm} l>0.
\]
Take $C=\min \{C_{0},C_{1}\} >0$ and $\sigma =\min \{\sigma_{0},\sigma_{1}\} >1$.
Thus, for every $g\in\mathcal{V}_{1}\cup\mathcal{V}_{2}$ and $x\in M$ there exists a point 
$y\in\mathcal{F}^{uu}_{1}(x,g)$ such that
\[
m\lbrace Dg^{n}_{\mid E^{c}(g^{l}(y))}\rbrace> C\sigma^{n}\hspace{0.5 cm}\text{for any}\hspace{0.5 cm} n>0,\hspace{0.5 cm} l>0.
\]
We may assume that $C=1$. Otherwise we take a fixed power of every $g\in\mathcal{V}_{1}\cup\mathcal{V}_{2}$.
Let $\mathcal{V}_{3}(f)\subset\mathcal{V}_{1}$ be a neighborhood of $f$ such that if $g\in\mathcal{V}_{3}$ then 
$g^{-1}\in\mathcal{V}_{2}$. Let $\lambda=\sigma^{-1}$, fix $0<\lambda <\lambda_{1} < 1$ and let $r >0$ be as
in Lemma~\ref{Lemma:Backward iterations contract}.
Consider $\epsilon >0$ given by Stable Manifold Theorem and let $r > 0$ be as above. Take $\delta > 0$
and take $\mathcal{V}_{4}(f)\subset\mathcal{V}_{3}$ a neighborhood of $f$ as in 
Lemma~\ref{Lemma:Consequence of the Stable Manifold Theorem}.
Since $f$ is transitive there exists a point $z\in M$ such that $\{ f^{n}(z); n\in\mathbb{N}\}$ and
$\{ f^{-n}(z); n\in\mathbb{N}\}$ are dense in $M$. Therefore 
\[
M=\underset{n\in \mathbb{N}}{\bigcup} B(f^{n}(z),\frac{\delta}{2})
\]
and by compactness there exist positive integers $n_{1} <\dots < n_{l}$ such that 

\[
\overset{l}{\underset{i=1}{\bigcup}} B(f^{n_{i}}(z),\frac{\delta}{2})=M.
\]
Next, choose a positive integer $m_{0}$ and a neighborhood $\mathcal{V}_{5}(f)\subset\mathcal{V}_{4}$
such that if $m\ge m_{0}$, $g\in \mathcal{V}_{5}$ and $q \in M$ then
\begin{align*}
\bullet&\hspace{0.5 cm} g^{m}(\mathcal{F}^{ss}_{\epsilon}(q,g))\subset B(g^{m}(q),\frac{\delta}{6})\\
\bullet&\hspace{0.5 cm} g^{-m}(\mathcal{F}^{uu}_{\epsilon}(q,g))\subset B(g^{-m}(q),\frac{\delta}{6})
\end{align*}

\begin{Affirmation}
For each $i=2,\dots ,l$ there exists $m_{i}\in\mathbb{Z}^{\ast}_{+}$ satisfying:\\

$\text{(i)}\hspace{0.3 cm}f^{m_{i}}(z)\in B(f^{n_{i}}(z),\frac{\delta}{6})\hspace{0.4 cm}\text{for}\hspace{0.4 cm} i=2,\dots,l$\\

$\text{(ii)}\hspace{0.3 cm}m_{2} > n_{1}+m_{0}$\\
$\phantom{\hspace{1.2 cm}}  m_{i} > m_{i-1}+m_{0}\hspace{0.4 cm}\text{for}\hspace{0.4 cm} i=3,\dots, l$
\end{Affirmation}
\begin{proof}
It follows by density of $\{f^{n}(z); n\in\mathbb{N}\}$ in $M$.
\end{proof}
\begin{Affirmation}
For each $i=2,\dots ,l$ there exists $\overset{\_\_\_\_}{m_{i}}\in\mathbb{Z}^{\ast}_{-}$ satisfying:\\

$\text{(iii)}\hspace{0.3 cm}f^{\overset{\_\_\_\_}{m_{i}}}(z)\in B(f^{n_{i}}(z),\frac{\delta}{6})\hspace{0.4 cm}\text{for}\hspace{0.4 cm} i=2,\dots,l$\\

$\text{(iv)}\hspace{0.3 cm}\overset{\_\_\_\_}{m_{2}} < n_{1}-m_{0}$\\
$\phantom{\hspace{1.3 cm}}  \overset{\_\_\_\_}{m_{i}} < \overset{\_\_\_\_\_\_\_\_\_}{m_{i-1}}-m_{0}\hspace{0.4 cm}\text{for}\hspace{0.4 cm} i=3,\dots, l$
\end{Affirmation}
\begin{proof}
It follows by density of $\{f^{-n}(z); n\in\mathbb{N}\}$ in $M$.
\end{proof}
\noindent Set $l_{0}=\max \{n_{l},m_{2},m_{3},m_{4},\dots ,m_{l},-\overset{\_\_\_\_}{m_{2}},-\overset{\_\_\_\_}{m_{3}},\dots,-\overset{\_\_\_\_}{m_{l}}\}$.\\
Observe that $l_{0}\ge n_{l} > n_{i}$ for $i=1,\dots, l-1$.\\
Take a neighborhood $\mathcal{V}(f)\subset\mathcal{V}_{5}$ such that $d_{C^{0}}(g^{n},f^{n}) < \frac{\delta}{6}$, 
for any $n\in\mathbb{Z}$ with $\lvert n\rvert \le l_{0}$, for any $g\in\mathcal{V}$.\\
\textbf{Step 2}\\
We will prove that any $g\in\mathcal{V}$ is transitive.
Take two arbitrary open sets $\mathcal{U},\mathcal{W}\subset M$.
Let us prove that there exists a positive integer
$k_{0}$ such that $g^{k_{0}}(\mathcal{U})\cap\mathcal{W}\neq \varnothing$.
Let $u\in\mathcal{U}$ and $w\in\mathcal{W}$. Let $\beta > 0$ be such that
$\mathcal{F}^{uu}_{\beta}(u,g)\subset\mathcal{U}$ and
$\mathcal{F}^{uu}_{\beta}(w,g^{-1})\subset\mathcal{W}$.
Take $n_{0}$ such that
$g^{n_{0}}(\mathcal{F}^{uu}_{\frac{\beta}{2}}(u,g))\supset\mathcal{F}^{uu}_{1}(g^{n_{0}}(u),g)$
and
$g^{-n_{0}}(\mathcal{F}^{uu}_{\frac{\beta}{2}}(w,g^{-1}))\supset\mathcal{F}^{uu}_{1}(g^{-n_{0}}(w),g^{-1})$.
Consider
$y\in\mathcal{F}^{uu}_{1}(g^{n_{0}}(u),g)$ and $x\in\mathcal{F}^{uu}_{1}(g^{-n_{0}}(w),g^{-1})$
satisfying:

\begin{equation}\label{Primeradoble}
\left\{\begin{aligned}
m&\{Dg^{n}_{\mid E^{c}}(g^{l}(y))\} > \sigma^{n}\hspace{1.25 cm} \text{for any} \qquad n>0,\quad l>0\\
m&\{Dg^{-n}_{\mid E^{c}}(g^{-l}(x))\} > \sigma^{n}\hspace{1.0 cm} \text{for any} \qquad n>0,\quad l>0.
\end{aligned}
\right.
\end{equation}
Observe that
\begin{align*}
\bullet&\hspace{0.2 cm} \mathcal{F}^{uu}_{\frac{\beta}{2}}(g^{-n_{0}}(y),g)\subset\mathcal{U}
\hspace{0.55 cm}\text{because}\hspace{0.1 cm} g^{-n_{0}}(y)
\in\mathcal{F}^{uu}_{\frac{\beta}{2}}(u,g)\subset\mathcal{F}^{uu}_{\beta}(u,g)\subset\mathcal{U}\\
\bullet&\hspace{0.2 cm} \mathcal{F}^{uu}_{\frac{\beta}{2}}(g^{n_{0}}(x),g^{-1})\subset\mathcal{W}
\hspace{0.2 cm}\text{because}\hspace{0.1 cm} g^{n_{0}}(x)
\in\mathcal{F}^{uu}_{\frac{\beta}{2}}(w,g^{-1})\subset\mathcal{F}^{uu}_{\beta}(w,g^{-1})\subset\mathcal{W}.
\end{align*}
Thus, there exist $A\subset\mathcal{U}$ a neighborhood of $g^{-n_{0}}(y)$ and $B\subset\mathcal{W}$
a neighborhood of $g^{n_{0}}(x)$ such that

\begin{equation}\label{Segundadoble}
\left\{\begin{aligned}
\hspace{0.1 cm}&\mathcal{F}^{uu}_{\frac{\beta}{2}}(a,g)\subset\mathcal{U}\hspace{1.5 cm} 
\text{for any}\hspace{0.5 cm} a\in A\\
\hspace{0.1 cm}&\mathcal{F}^{uu}_{\frac{\beta}{2}}(b,g^{-1})\subset\mathcal{W}\hspace{1.0 cm} 
\text{for any}\hspace{0.5 cm} b\in B.
\end{aligned}
\right.
\end{equation}

Let $\eta >0$ be such that

\begin{equation}\label{Terceradoble}
\left\{\begin{aligned}
\hspace{0.1 cm}&g^{-n_{0}}(\mathcal{W}^{cu}_{\eta}(y,g))\subset A\subset\mathcal{U}\\
\hspace{0.1 cm}&g^{n_{0}}(\mathcal{W}^{cu}_{\eta}(x,g^{-1}))\subset B\subset\mathcal{W}
\end{aligned}
\right.
\end{equation}
Next, choose a positive integer $m'$ such that $\lambda^{m'}_{1}r < \eta$ and

\begin{equation}\label{Cuartadoble}
\left\{\begin{aligned}
\hspace{0.1 cm}&g^{m'+n_{0}}(\mathcal{F}^{uu}_{\frac{\beta}{2}}(q,g))\supset
\mathcal{F}^{uu}_{\epsilon}(g^{m'+n_{0}}(q),g)\hspace{0.4 cm}\text{for any}\hspace{0.4 cm} q\in M\\
\hspace{0.1 cm}&g^{-(m'+n_{0})}(\mathcal{F}^{uu}_{\frac{\beta}{2}}(q,g^{-1}))\supset
\mathcal{F}^{uu}_{\epsilon}(g^{-(m'+n_{0})}(q),g^{-1})\hspace{0.4 cm}\text{for any}\hspace{0.4 cm} q\in M
\end{aligned}
\right.
\end{equation}

Set $k'=n_{0}+m'$. Thus, using \eqref{Primeradoble}, we get

\begin{align*}
\hspace{0.1 cm}&\prod_{j=0}^{n-1}\parallel Dg^{-1}_{\mid_{E^{c}(g^{-j}(y)})}
\parallel<\lambda^{n},\hspace{1 cm} 0\leq n\leq m'\\
\hspace{0.1 cm}&\prod_{j=0}^{n-1}\parallel Dg_{\mid_{E^{c}(g^{j}(x)})}
\parallel<\lambda^{n},\hspace{1 cm} 0\leq n\leq m'
\end{align*}

and therefore

\begin{align*}
\hspace{0.1 cm}&\prod_{j=0}^{n-1}\parallel Dg^{-1}_{\mid_{E^{cu}(g^{-j}(y)})}
\parallel<\lambda^{n},\hspace{1 cm} 0\leq n\leq m'\\
\hspace{0.1 cm}&\prod_{j=0}^{n-1}\parallel Dg_{\mid_{E^{cu}(g^{j}(x)})}
\parallel<\lambda^{n},\hspace{1 cm} 0\leq n\leq m'
\end{align*}

From   Lemma~\ref{Lemma:Backward iterations contract} we conclude that

\begin{equation}\label{Quintadoble}
\left\{\begin{aligned}
\hspace{0.1 cm}&g^{-m'}(\mathcal{W}^{cu}_{r}(g^{m'}(y),g))\subset
\mathcal{W}^{cu}_{\lambda^{m'}_{1}r}(y,g)\subset
\mathcal{W}^{cu}_{\eta}(y,g)\\
\hspace{0.1 cm}&g^{m'}(\mathcal{W}^{cu}_{r}(g^{-m'}(x),g^{-1}))\subset
\mathcal{W}^{cu}_{\lambda^{m'}_{1}r}(x,g^{-1})\subset
\mathcal{W}^{cu}_{\eta}(x,g^{-1})
\end{aligned}
\right.
\end{equation}

and hence, using \eqref{Terceradoble}, we have

\begin{equation}\label{Sextadoble}
\left\{\begin{aligned}
\hspace{0.1 cm}&g^{-k'}(\mathcal{W}^{cu}_{r}(g^{m'}(y),g))\subset A\subset\mathcal{U}\\
\hspace{0.1 cm}&g^{k'}(\mathcal{W}^{cu}_{r}(g^{-m'}(x),g^{-1}))\subset B\subset\mathcal{W}.
\end{aligned}
\right.
\end{equation}
Particularly, it follows

\begin{equation}\label{Septimadoble}
\left\{\begin{aligned}
\hspace{0.1 cm}&\mathcal{W}^{cu}_{r}(g^{m'}(y),g)\subset g^{k'}(\mathcal{U})\\
\hspace{0.1 cm}&\mathcal{W}^{cu}_{r}(g^{-m'}(x),g^{-1})\subset g^{-k'}(\mathcal{W}).
\end{aligned}
\right.
\end{equation}
Moreover, if $p\in\mathcal{W}^{cu}_{r}(g^{m'}(y),g),\hspace{0.1 cm} q\in\mathcal{W}^{cu}_{r}(g^{-m'}(x),g^{-1})$
then $g^{-k'}(p)\in A$ and $g^{k'}(q)\in B$ due to \eqref{Sextadoble}.
Thus, from \eqref{Segundadoble},

\[
 \mathcal{F}^{uu}_{\frac{\beta}{2}}(g^{-k'}(p),g)\subset\mathcal{U}
\hspace{0.5 cm}\text{and}\hspace{0.5 cm}\mathcal{F}^{uu}_{\frac{\beta}{2}}(g^{k'}(q),g^{-1})\subset\mathcal{W}
\]
 
and hence,  \eqref{Cuartadoble} imply

\begin{equation}\label{Octavadoble}
\left\{\begin{aligned}
\hspace{0.1 cm}&\mathcal{F}^{uu}_{\epsilon}(p,g)\subset
g^{k'}(\mathcal{F}^{uu}_{\frac{\beta}{2}}(g^{-k'}(p),g))\subset g^{k'}(\mathcal{U})\\
\hspace{0.1 cm}&\mathcal{F}^{uu}_{\epsilon}(q,g^{-1})\subset
g^{-k'}(\mathcal{F}^{uu}_{\frac{\beta}{2}}(g^{k'}(q),g^{-1}))\subset g^{-k'}(\mathcal{W}).
\end{aligned}
\right.
\end{equation}
Finally, from \eqref{Septimadoble} and \eqref{Octavadoble} we conclude that

\begin{align*}
(i)&\hspace{0.2 cm}\mathcal{W}^{cu}_{r}(g^{m'}(y),g)\subset g^{k'}(\mathcal{U})\\
(ii)&\hspace{0.2 cm}\mathcal{F}^{uu}_{\epsilon}(p,g)\subset g^{k'}(\mathcal{U}),
\hspace{0.1 cm}\forall p\in\mathcal{W}^{cu}_{r}(g^{m'}(y),g)\\
(iii)&\hspace{0.2 cm}\mathcal{W}^{cs}_{r}(g^{-m'}(x),g)\subset g^{-k'}(\mathcal{W})\\
(iv)&\hspace{0.2 cm}\mathcal{F}^{ss}_{\epsilon}(q,g)\subset g^{-k'}(\mathcal{W}),
\hspace{0.1 cm}\forall q\in\mathcal{W}^{cs}_{r}(g^{-m'}(x),g)
\end{align*}
For the sake of simplicity, we will denote $g^{m'}(y)$ for $\overset{\_}{y}$
and $g^{-m'}(x)$ for $\overset{\_}{x}$.\\
\noindent Since $M=\overset{l}{\underset{i=1}{\cup}}B(f^{n_{i}}(z),\frac{\delta}{2})$, 
there are $i,j\in\{1,\dots,l\}$ such that
\[
\overset{\_}{y}\in B(f^{n_{i}}(z),\frac{\delta}{2})\hspace{0.2 cm}\text{and}\hspace{0.2 cm}
\overset{\_}{x}\in B(f^{n_{j}}(z),\frac{\delta}{2}).
\] 
\textbullet\hspace{0.3 cm} Case $i=j$

In this case, $d(\overset{\_}{x},\overset{\_}{y})<\delta$. Thus, using 
Lemma~\ref{Lemma:Consequence of the Stable Manifold Theorem}, 
$\mathcal{F}^{uu}_{\epsilon}(\overset{\_}{y},g)\pitchfork W^{cs}_{r}(\overset{\_}{x},g)\neq\varnothing$.
Moreover, by $(ii)$ and by $(iii)$, we have that 
$\mathcal{F}^{uu}_{\epsilon}(\overset{\_}{y},g)\subset g^{k'}(\mathcal{U})$ and
$\mathcal{W}^{cs}_{r}(\overset{\_}{x},g)\subset g^{-k'}(\mathcal{W})$, and
hence,  $g^{k'}(\mathcal{U})\cap g^{-k'}(\mathcal{W})\neq\emptyset$, i.e.,
$g^{2k'}(\mathcal{U})\cap \mathcal{W}\neq\emptyset$.

Next, we will prove the case $i<j$. The case $i>j$ is similar.\\
\textbullet\hspace{0.3 cm} Case $i<j$\\
\noindent First assume $i > 1$. Consider $j=i+k$ for $k=1,2,\dots,l-i$. In 
this case we have that

\begin{align*}
d(\overset{\_}{y},g^{m_{i}}(z))&\leq d(\overset{\_}{y},f^{n_{i}}(z))+
d(f^{n_{i}}(z),f^{m_{i}}(z))+d(f^{m_{i}}(z),g^{m_{i}}(z))\\
\hspace{0.1 cm}&< \frac{\delta}{2}+\frac{\delta}{6}+\frac{\delta}{6} <\delta\\
\end{align*}

\noindent and therefore 
\[
\mathcal{F}^{ss}_{\epsilon}(g^{m_{i}}(z),g)\pitchfork W^{cu}_{r}(\overset{\_}{y},g)\neq\emptyset.
\]
Take $p\in\mathcal{F}^{ss}_{\epsilon}(g^{m_{i}}(z),g)\pitchfork W^{cu}_{r}(\overset{\_}{y},g)$. 
Since

\begin{align*}
m_{j}-m_{i}&=m_{i+k}-m_{i}=(m_{i+k}-m_{i+(k-1)})+(m_{i+(k-1)}-m_{i+(k-2)})\\
\hspace{0.1 cm}&+\dots + (m_{i+1}-m_{i})> km_{0} > m_{0},\\
g^{m_{j}-m_{i}}(\mathcal{F}^{ss}_{\epsilon}(g^{m_{i}}(z),g))&\subset B(g^{m_{j}}(z),\frac{\delta}{6}),\\
\end{align*}
and from this it follows that
\begin{equation*}
 g^{m_{j}-m_{i}}(p)\in B(g^{m_{j}}(z),\frac{\delta}{6}).
\end{equation*}

\noindent Thus,
\begin{align*}
d(g^{m_{j}-m_{i}}(p),\overset{\_}{x})&\leq d(g^{m_{j}-m_{i}}(p),g^{m_{j}}(z))+d(g^{m_{j}}(z),f^{m_{j}}(z))\\
\hspace{0.1 cm}&+d(f^{m_{j}}(z),f^{n_{j}}(z))+d(f^{n_{j}}(z),\overset{\_}{x})\\
\hspace{0.1 cm}&< \frac{\delta}{6}+\frac{\delta}{6}+\frac{\delta}{6}+\frac{\delta}{2}=\delta\\
\end{align*}
\noindent and from Lemma~\ref{Lemma:Consequence of the Stable Manifold Theorem}, we get

\[
(v)\hspace{0.5 cm} \mathcal{F}^{uu}_{\epsilon}(g^{m_{j}-m_{i}}(p),g)\pitchfork \mathcal{W}^{cs}_{r}(\overset{\_}{x},g)\neq\emptyset.
\]

\noindent Using that $(m_{j}-m_{i}) >0$ and $p\in \mathcal{W}^{cu}_{r}(\overset{\_}{y},g)$
and using $(ii)$, we have that

\[
 \mathcal{F}^{uu}_{\epsilon}(g^{m_{j}-m_{i}}(p),g)\subset 
g^{m_{j}-m_{i}}(\mathcal{F}^{uu}_{\epsilon}(p,g))\subset g^{m_{j}-m_{i}}(g^{k'}(\mathcal{U})).
\]
From $(iii)$ and $(v)$ we conclude that 

\[
g^{m_{j}-m_{i}}(g^{k'}(\mathcal{U}))\cap g^{-k'}(\mathcal{W})\neq \emptyset.
\]
\noindent In this case the proof is completed.\\
\newline
\noindent Now, assume $i=1$.\\
Consider $j=i+k$ for $k=1,2,\dots,l-i$.
\noindent In this case we have that

\begin{align*}
d(\overset{\_}{y},g^{n_{1}}(z))&\leq d(\overset{\_}{y},f^{n_{1}}(z))+d(f^{n_{1}}(z),g^{n_{1}}(z))\\
\hspace{0.1 cm}&<\frac{\delta}{2}+\frac{\delta}{6}< \delta\\
\end{align*}

\noindent and therefore

\[
\mathcal{F}^{ss}_{\epsilon}(g^{n_{1}}(z),g)\pitchfork \mathcal{W}^{cu}_{r}(\overset{\_}{y},g)\neq\emptyset.
\]
Take $p\in\mathcal{F}^{ss}_{\epsilon}(g^{n_{1}}(z),g)\pitchfork \mathcal{W}^{cu}_{r}(\overset{\_}{y},g)$.
Since

\begin{align*}
m_{j}-n_{1}=m_{1+k}-n_{1}&=(m_{1+k}-m_{k})+(m_{k}-m_{k-1})+\dots+(m_{2}-n_{1})\\
\hspace{0.1 cm}& >km_{0}>m_{0},\\
g^{m_{j}-n_{1}}(\mathcal{F}^{ss}_{\epsilon}(g^{n_{1}}(z),g))&\subset B(g^{m_{j}}(z),\frac{\delta}{6}),\\
\end{align*}
and from this it follows that

\[
g^{m_{j}-n_{1}}(p)\in B(g^{m_{j}}(z),\frac{\delta}{6}).
\]
Thus,
\begin{align*}
d(g^{m_{j}-n_{1}}(p),\overset{\_}{x})&\leq d(g^{m_{j}-n_{1}}(p),g^{m_{j}}(z))+d(g^{m_{j}}(z),f^{m_{j}}(z))\\
\hspace{0.1 cm}&+d(f^{m_{j}}(z),f^{n_{j}}(z))+d(f^{n_{j}}(z),\overset{\_}{x})\\
\hspace{0.1 cm}&< \frac{\delta}{6}+\frac{\delta}{6}+\frac{\delta}{6}+\frac{\delta}{2}= \delta\\
\end{align*}
and from Lemma~\ref{Lemma:Consequence of the Stable Manifold Theorem} , we get
\[
(vi)\hspace{0.5 cm}\mathcal{F}^{uu}_{\epsilon}(g^{m_{j}-n_{1}}(p),g)\pitchfork 
\mathcal{W}^{cs}_{r}(\overset{\_}{x},g)\neq\emptyset.
\]
However,
\[
\mathcal{F}^{uu}_{\epsilon}(g^{m_{j}-n_{1}}(p),g)\subset 
g^{m_{j}-n_{1}}(\mathcal{F}^{uu}_{\epsilon}(p,g))\subset g^{m_{j}-n_{1}}(g^{k'}(\mathcal{U})) 
\]
due to
\[
m_{j}-n_{1}>0,\hspace{0.2 cm} 
p\in\mathcal{W}^{cu}_{r}(\overset{\_}{y},g)\hspace{0.2 cm}\text{and}\hspace{0.2 cm}(ii).
\]
From $(iii)$ and $(vi)$ we conclude that
\[
g^{m_{j}-n_{1}}(g^{k'}(\mathcal{U}))\cap g^{-k'}(\mathcal{W})\neq\emptyset.
\]
Hence, the case $i<j$ is completed.
The case $i>j$ follows by symmetry, and the proof of Theorem is completed.

\end{proof}

\noindent The proof of last Theorem suggests the following Proposition as a possible, future, alternative way to remove
the condition Property SH for $f^{-1}$, to get robust transitivity.

\begin{Proposition}
Let $M$ be a compact Riemannian manifold,  $f$ a partially hyperbolic $C^{r}$-diffeomorphism in $M$,
and $p$ a periodic hyperbolic point for $f$, whose central direction is unstable and with 
a $\frac{\delta}{4}$-dense orbit, $\delta$ as in Lemma~\ref{Lemma:Consequence of the Stable Manifold Theorem}. If $f$  satisfy Property SH then $f$ is robustly transitive.
\end{Proposition}
\begin{proof}
Analogously to the proof in Theorem~\ref{Theorem:Third result} we divide the proof in two steps. The first step
deals with the construction of an appropriate neighbourhood $\mathcal{V}$ of $f$. 
In the second step we prove that any diffeomorphism in $\mathcal{V}$ is transitive.\\
\textbf{Step 1}\\
From Theorem~\ref{Theorem:SH is robust} there exist a $C^{1}$-neighbourhood,  $\mathcal{V}_{1}(f)$, $C>0$ and 
$\sigma>1$ such that for every $g\in\mathcal{V}_{1}$ and $x\in M$ there exists a point 
$y\in\mathcal{F}^{uu}_{1}(x,g)$ such that 
\begin{equation}\label{Cuarta}
m\lbrace Dg^{n}_{\mid E^{c}(g^{l}(y))}\rbrace> C\sigma^{n}\hspace{0.5 cm}\text{for any}\hspace{0.5 cm} n>0,\hspace{0.5 cm} l>0.
\end{equation}
We may assume that $C=1$. Otherwise we take a fixed power of every $g\in\mathcal{V}_{1}$.
Let $\lambda=\sigma^{-1}$, fix $0<\lambda <\lambda_{1} < 1$ and let $r >0$ be as
in Lemma~\ref{Lemma:Backward iterations contract}.
Consider $\epsilon >0$ given by Stable Manifold Theorem and let $r > 0$ be as above. Take $\delta > 0$
and $\mathcal{V}_{2}(f)\subset\mathcal{V}_{1}$ a neighbourhood of $f$ as in 
Lemma~\ref{Lemma:Consequence of the Stable Manifold Theorem}.
Next choose a neighbourhood $\mathcal{V}(f)$ contained in $\mathcal{V}_{2}$
such that if $g\in\mathcal{V}$ then the hyperbolic continuation $p_{g}$
of $p$  is a hyperbolic periodic point of $g$ with unstable central direction
and a $\frac{\delta}{2}$-dense orbit.

\noindent\textbf{Step 2}\\
We will prove that any $g\in\mathcal{V}$ is transitive.
Take two arbitrary open sets $\mathcal{U},\mathcal{W}\subset M$.
Let us prove that there exists a positive integer
$k_{0}$ such that $g^{k_{0}}(\mathcal{U})\cap\mathcal{W}\neq \varnothing$.
Choose $u\in\mathcal{U}$ and $x\in\mathcal{W}$. Let $\beta > 0$ be such that
$\mathcal{F}^{uu}_{\beta}(u,g)\subset\mathcal{U}$ and
$\mathcal{F}^{ss}_{\beta}(x,g)\subset\mathcal{W}$.
Take $n_{0}$ such that
$g^{n_{0}}(\mathcal{F}^{uu}_{\frac{\beta}{2}}(u,g))\supset\mathcal{F}^{uu}_{1}(g^{n_{0}}(u),g)$.
Consider $y\in\mathcal{F}^{uu}_{1}(g^{n_{0}}(u),g)$ satisfying:

\begin{equation}\label{Primeradoble}
m\{Dg^{n}_{\mid E^{c}}(g^{l}(y))\} > \sigma^{n}\hspace{1.25 cm} \text{for any} \qquad n>0,\quad l>0\\
\end{equation}

\noindent Let $\eta >0$ be such that

\begin{equation}\label{Terceradoble}
g^{-n_{0}}(\mathcal{W}^{cu}_{\eta}(y,g))\subset\mathcal{U}\\
\end{equation}
Next, choose a positive integer $m'$ such that $\lambda^{m'}_{1}r < \eta$. 
Set $k'=n_{0}+m'$. Thus, using \eqref{Primeradoble}, we get

\begin{equation}
\prod_{j=0}^{n-1}\parallel Dg^{-1}_{\mid_{E^{c}(g^{-j}(y)})}
\parallel<\lambda^{n},\hspace{1 cm} 0\leq n\leq m'\\
\end{equation}

\noindent and therefore

\begin{equation}
\prod_{j=0}^{n-1}\parallel Dg^{-1}_{\mid_{E^{cu}(g^{-j}(y)})}
\parallel<\lambda^{n},\hspace{1 cm} 0\leq n\leq m'\\
\end{equation}

\noindent From   Lemma~\ref{Lemma:Backward iterations contract} we conclude that

\begin{equation}\label{Quintadoble}
g^{-m'}(\mathcal{W}^{cu}_{r}(g^{m'}(y),g))\subset
\mathcal{W}^{cu}_{\lambda^{m'}_{1}r}(y,g)\subset
\mathcal{W}^{cu}_{\eta}(y,g)\\
\end{equation}

\noindent and hence, using \eqref{Terceradoble}, we have

\begin{equation}\label{Sextadoble}
g^{-k'}(\mathcal{W}^{cu}_{r}(g^{m'}(y),g))\subset\mathcal{U}\\
\end{equation}
\noindent Particularly, it follows

\begin{equation}\label{Septimadoble}
\mathcal{W}^{cu}_{r}(g^{m'}(y),g)\subset g^{k'}(\mathcal{U})\\
\end{equation}

\noindent For the sake of simplicity, we will denote $g^{m'}(y)$ for $\overset{\_}{y}$
and $g^{-t}(x)$ for $\overset{\_}{x}$. Now choose a positive integer $t$
such that
\begin{equation}
\mathcal{F}_{\epsilon}^{ss}(g^{-t}(w),g)\subset g^{-t}(\mathcal{F}_{\beta}^{ss}(w,g))
\end{equation}

\noindent Let $L$ be the period of $p_{g}$. There exist $i,j\in\{0,1,\cdots,L-1\}$
such that $\overset{\_}{y}\in B(g^{i}(p_{g}),\frac{\delta}{2})$ and
$\overset{\_}{x}\in B(g^{j}(p_{g}),\frac{\delta}{2})$.

\noindent\textbullet\hspace{0.3 cm} Case $i=j$

\noindent In this case, $d(\overset{\_}{x},\overset{\_}{y})<\delta$. Thus, using 
Lemma~\ref{Lemma:Consequence of the Stable Manifold Theorem}, 
$\mathcal{F}^{ss}_{\epsilon}(\overset{\_}{x},g)\pitchfork W^{cu}_{r}(\overset{\_}{y},g)\neq\varnothing$.
Moreover, we have $\mathcal{F}^{ss}_{\epsilon}(\overset{\_}{x},g)\subset g^{-t}(W)$
and $W^{cu}_{r}(\overset{\_}{y},g)\subset g^{k'}(\mathcal{U})$ so
$g^{-t}(\mathcal{W})\cap g^{k'}(\mathcal{U})\neq\varnothing$ that is
$\mathcal{W}\cap g^{k'+t}(\mathcal{U})\neq\varnothing$.
\newline

\noindent Next, we will prove the case $i\neq j$. \\
\textbullet\hspace{0.3 cm} Case $i\neq j$\\
\noindent We may assume without loss of generality that $i < j$. Consider $j=i+k$ for some $k\in\{1,2,\dots,L-i\}$. 
Hence  $d(\overset{\_}{y},g^{i}(p_{g}))< \frac{\delta}{2}$ and 
$\mathcal{F}^{ss}_{\epsilon}(g^{i}(p_{g}),g)\pitchfork W^{cu}_{r}(\overset{\_}{y},g)\neq\emptyset$.
Take $q\in\mathcal{F}^{ss}_{\epsilon}(g^{i}(p_{g}),g)\pitchfork W^{cu}_{r}(\overset{\_}{y},g)$. 
Note that $g^{j-i}(W_{r}^{cu}(\overset{\_}{y},g))$ intersect transversally $\mathcal{F}^{ss}(g^{i}(p_{g}),g)$ in $g^{j-i}(q)$.
Using the Lambda-Lemma for $g^{L}$
we get the existence of $N_{0}\in\mathbb{N}$ such that $g^{nL}(g^{j-i}(D))$ is $C^{1}$-close to
$W_{r}^{cu}((g^{i}(p_{g})),g)$ for any $n\geq N_{0}$ where $D$ is a disc such that 
$q\in D\subset W_{r}^{cu}(\overset{\_}{y},g)$.

\noindent Also as $\overset{\_}{x}\in B(g^{i}(p_{g}),\frac{\delta}{2})$ then
$\mathcal{F}_{\epsilon}^{ss}(\overset{\_}{x},g)\pitchfork W_{r}^{cu}(g^{j}(p_{g}),g)\neq\varnothing$.
So there exists $N_{1}\in\mathbb{N}$ with $N_{1}>N_{0}$ such that
$g^{nL}(g^{j-i}(D))\pitchfork\mathcal{F}_{\epsilon}^{ss}(\overset{\_}{x},g)\neq\varnothing$
for any $n\geq N_{1}$. So $g^{nL}(g^{j-i}(g^{k'}(\mathcal{U})))\cap g^{-t}(\mathcal{W})\neq\varnothing$
for any $n\geq N_{1}$. Consequently $g$ is transitive.

\end{proof}

%% file: PropertySHandDensityofPeriodicPoints.tex
\section{Property SH and Density of Periodic Points}
\label{sec:Property SH and Density of Periodic Points}

Here we prove that for a diffeomorphism exhibiting Property SH
and minimality of the  strong stable foliation the set of its 
periodic points is dense. So both transitivity and density
of the periodic points are robust properties under the hypotheses
of Property SH and minimality of the strong stable foliation.

\begin{Theorem}\label{Theorem:Density of periodic points}
Let $f\in Diff^{r}(M)$  be a partially hyperbolic diffeomorphism exhibiting Property SH
and such that the strong stable foliation is minimal. Then, $\overline{Per(f)}=M$. 
\end{Theorem}

\begin{proof}

\begin{Remark}
Changing
$f$ by a power of itself, we can assume that there is $\sigma >1$ such that for any
$x\in M$ there exists $y^{u}\in\mathcal{F}^{uu}_{1}(x,f)$ such that
\begin{equation}
m\lbrace Df^{n}_{\mid E^{c}(f^{l}(y^{u}))}\rbrace> \sigma^{n}\hspace{0.2 cm}
\text{for any}\hspace{0.2 cm} n>0,\hspace{0.2 cm} l>0.
\end{equation}

\end{Remark}

Let SH be defined by:
\begin{equation}
SH=\{y\in M:m\lbrace Df^{n}_{\mid E^{c}(f^{l}(y))}\rbrace> \sigma^{n}\hspace{0.2 cm}
\text{for any}\hspace{0.2 cm} n>0,\hspace{0.2 cm} l>0\}. 
\end{equation}

\begin{Lemma}\label{Lemma:Reducing the theorem }
If $SH\subset\overline{Per(f)}$ then $M\subset\overline{Per(f)}$. 
\end{Lemma}

\begin{proof}
Given $x\in M$ and $V$ an open set containing $x$ choose $\beta >0$ be such that
$\mathcal{F}^{uu}_{\beta}(x,f)\subset V$ and $l_{0}$ such that
$\mathcal{F}^{uu}_{1}(f^{l_{0}}(x),f)\subset f^{l_{0}}(\mathcal{F}^{uu}_{\beta}(x,f))$.
Then take $h\in\mathcal{F}^{uu}_{1}(f^{l_{0}}(x),f)\cap SH$ and use continuity of $f$.
\end{proof}

\noindent From now on our goal will be to prove that $SH\subset\overline{Per(f)}$.

\noindent Let us fix $\epsilon', r'$ and $d'$
as in Lemma~\ref{Lemma:From the Stable Manifold Theorem}.

\begin{Definition}\label{Definition:Cylinders}
We will call a cylinder any open set $W\subset M$, with $diam (W)< d'$, which is the domain of some
local chart $\eta:M\rightarrow\mathbb{R}^{n}$ trivializing the strong stable foliation
such that $W^{cu}_{r'}(y,f)\nsubseteq W$ and
$\mathcal{F}^{ss}_{\epsilon'}(y,f)\nsubseteq W$ for any $y\in W$.  
\end{Definition}

\begin{Lemma}\label{Lemma:Existence of cylinders} 
For every $x\in M$ there exists a cylinder containing $x$. 
\end{Lemma}

\begin{proof}
First observe that there exists a local chart $(\widetilde{W},\widetilde{\eta})$, 
trivializing the strong stable foliation, with $x\in\widetilde{W}$ and such that
$W^{cu}_{r'}(x,f)\nsubseteq\widetilde{W},\mathcal{F}^{ss}_{\epsilon'}(x,f)\nsubseteq\widetilde{W}$.
Now by the continuous dependence of the manifolds $W^{cu}_{r'}(y,f)$ and
$\mathcal{F}^{ss}_{\epsilon'}(y,f)$
on the point $y$, follows the existence of an open set $\widetilde{\widetilde{W}}\subset\widetilde{W}$
containing $x$ and such that 
$W^{cu}_{r'}(z,f)\nsubseteq\widetilde{W},\mathcal{F}^{ss}_{\epsilon'}(z,f)\nsubseteq\widetilde{W}$
for any $z\in\widetilde{\widetilde{W}}$. Finally take a local chart trivializing the strong stable foliation $(W,\eta)$, 
with $x\in W\subset\widetilde{\widetilde{W}}$ and $diam (W)< d'$. 
\end{proof}

Notice that there exists a base $B$ of open sets of $M$
whose elements are cylinders. Let $\mathcal{C}$ be an open covering of cylinders of the manifold $M$ and 
$L$ its Lebesgue number.

\begin{Lemma}\label{Lemma:Local projections}
Let $C$ be a cylinder and let $\eta:C\rightarrow U^{cu}\times V^{ss}$ 
be a local chart trivializing the strong stable foliation, where $U^{cu}$, 
$V^{ss}$ are open sets in $\mathbb{R}^{c+uu},\mathbb{R}^{ss}$ respectively
and $0\in\eta(C)$. Let 
$\pi:U^{cu}\times V^{ss}\rightarrow U^{cu}\times\{0\}$ be the projection of
$\mathbb{R}^{c+uu+ss}$ on $\mathbb{R}^{c+uu}\times\{0\}$. Let $h\in C$ and $\hat{r}>0$
be such that $\overline{\mathcal{F}^{ss}_{\hat{r}}(h,f)}\subset C$ and let $\hat{\delta} >0$ be
such that $\overline{W^{cu}_{\hat{\delta}}(y,f)}\subset C$ for any 
$y\in\overline{\mathcal{F}^{ss}_{\hat{r}}(h,f)}$. Denote $\eta(h)$ by
$(h_{cu},h_{ss})$. Then the following
hold:

(i) $\pi_{\mid\eta(\overline{W^{cu}_{\hat{\delta}}(y,f)})}$ is an
homeomorphism on its image for any $y\in\overline{\mathcal{F}^{ss}_{\hat{r}}(h,f)}$.

(ii) There exists an open ball $B\subset U^{cu}$ centered at $h_{cu}$ such that
\begin{equation*}
B\times\{0\}\subset\bigcap_{y\in\overline{\mathcal{F}^{ss}_{\hat{r}}(h,f)}}\pi(\eta(W^{cu}_{\hat{\delta}}(y,f))).
\end{equation*}

(iii) There exists $0<\overline{\delta}<\hat{\delta}$ such that for any
$y_{1},y_{2}\in\overline{\mathcal{F}^{ss}_{\hat{r}}(h,f)}$ there exists
a continuous map $\pi_{y_{1},y_{2}}:W^{cu}_{\overline{\delta}}(y_{1},f)\rightarrow W^{cu}_{\hat{\delta}}(y_{2},f)$
and if $t'=\pi_{y_{1},y_{2}}(t)$ then $t'\in\mathcal{F}^{ss}(t,f)$.

\end{Lemma}
\begin{proof}
$(i)$ Observe that as $C$ is a cylinder if
$\overline{W^{cu}_{\overline{r}}(y,f)}\cap\mathcal{F}^{ss}_{\overline{\epsilon}}(x,f)\neq\varnothing$
then this intersection is exactly one point and $\pi_{\mid\eta(\overline{W^{cu}_{\overline{r}}(y,f)})}$
is injective.

$(ii)$ As $\pi_{\mid\eta(\overline{W^{cu}_{\hat{\delta}}(y,f)})}$ is an homeomorphism
on its image by the Invariance of Domain Theorem 
for any $y\in\overline{\mathcal{F}^{ss}_{\hat{r}}(h,f)}$
there exists an open ball $B_{y}\subset U^{cu}$ centered at $h_{cu}$ such that
$B_{y}\times\{0\}\subset\pi(\eta(W^{cu}_{\hat{\delta}}(y,f)))$.
Now use the compacity of $\overline{\mathcal{F}^{ss}_{\hat{r}}(h,f)}$ and  
continuous dependence of the manifolds $W^{cu}_{\hat{\delta}}(y,f)$ on the points $y$.

$(iii)$ Choose $\overline{\delta}<\hat{\delta}$ such that
$\eta(W^{cu}_{\overline{\delta}}(y,f))\subset B\times V^{ss}$ for any 
$y\in\overline{\mathcal{F}^{ss}_{\hat{r}}(h,f)}$.
Then define:
\begin{equation*}\pi_{y_{1},y_{2}}=(\eta)^{-1}\circ(\pi_{\mid\eta(\overline{W^{cu}_{\hat{\delta}}(y_{2},f)})})^{-1}
\circ\pi_{\mid\eta(\overline{W^{cu}_{\hat{\delta}}(y_{1},f)})}\circ\eta:W^{cu}_{\overline{\delta}}(y_{1},f)
\rightarrow W^{cu}_{\hat{\delta}}(y_{2},f)
\end{equation*} 

\end{proof}

\noindent Choose $\delta >0$ such that if $dist (z, SH)< \delta$ then
\begin{equation}\label{SH ' contraction}
\parallel Df^{-1}_{\mid E^{c}(f(z))} \parallel< (\sigma')^{-1}<1
\end{equation}
for some $1< \sigma'<\sigma$.

\noindent Let us define the set $\text{SH}'$ by 
\begin{equation*}
\text{SH}'=\bigcup_{z\in \text{SH}}\mathcal{F}^{ss}_{\delta}(z,f).
\end{equation*}

\begin{Lemma}\label{Lemma:Expansion for SH'}
If $x\in SH'$ then $m\lbrace Df^{n}_{\mid E^{c}(f^{l}(x))}\rbrace> (\sigma')^{n}$ for any $n>0, l>0$.  
\end{Lemma}
\begin{proof}
It follows by induction, using \eqref{SH ' contraction} and the fact that $f(\text{SH})\subset \text{SH}$.
\end{proof}

\noindent Let $\alpha=(\sigma')^{-1}$, fix $\alpha_{1}$ with $0< \alpha<\alpha_{1}<1$ and
let $r_{0}$ be as in Lemma~\ref{Lemma:Backward iterations contract}.

\noindent Consider $\lambda<1$ the contraction factor of the strong stable subbundle. 

\noindent Let $h\in\text{SH}$ and let $U\subset M$ be an open set containing $h$. We will prove
that there exists a periodic point in $U$.

\noindent Let $C\in B$ be a cylinder contained in $U$ such that $h\in C$.

\noindent Take $\hat{r}, K$ and $n_{0}\in\mathbb{N}$ such that  $0<\hat{r}<\delta$, $\overline{\mathcal{F}^{ss}_{2\hat{r}}(h,f)}\subset C$,
$K> \delta+\epsilon_{1}$, 
\begin{equation}\label{the d-minimality}
\mathcal{F}^{ss}_{K}(x,f)\pitchfork W^{cu}_{d}(y,f)\neq\varnothing, \forall x,y\in M.
\end{equation}
and
\begin{equation*}
 \mathcal{F}^{ss}_{K}(f^{-n_{0}}(x),f)\subset f^{-n_{0}}(\mathcal{F}^{ss}_{\hat{r}}(x,f)),\quad\forall x\in M.  
\end{equation*}

Take also $\hat{\delta}$ satisfying simultaneously the following three conditions:\\

1)$\overline{W^{cu}_{\hat{\delta}}(y,f)}\subset C,\quad\forall y\in\overline{\mathcal{F}^{ss}_{\hat{r}}(h,f)}$

2)$f^{-n_{0}}(W^{cu}_{\hat{\delta}}(y,f))\subset W^{cu}_{d}(f^{-n_{0}}(y),f),\quad\forall y\in\overline{\mathcal{F}^{ss}_{\hat{r}}(h,f)}$

3)If $dist(z,\overline{\mathcal{F}^{ss}_{\hat{r}}(h,f)})< \hat{\delta}$ then $\overline{\mathcal{F}^{ss}_{\hat{r}}(z,f)}\subset C$.\\

\noindent Let now $\overline{\delta}$ and $\pi_{y_{1},y_{2}}$ be like in Lemma~\ref{Lemma:Local projections} and 
such that
\begin{equation}\label{inclusions}
\mathcal{F}^{ss}_{\hat{r}}(\pi_{h,y_{2}}(q),f)\subset \mathcal{F}^{ss}_{2\hat{r}}(q,f),
\quad\forall q\in W^{cu}_{\overline{\delta}}(h,f),
\quad\forall y_{2}\in \overline{\mathcal{F}^{ss}_{\hat{r}}(h,f)} 
\end{equation}

\noindent Let $N\in\mathbb{N}$ be such that $(\alpha_{1})^{N}r_{0}< \overline{\delta}$ and
$f^{N}(\overline{\mathcal{F}^{ss}_{2\hat{r}}(y,f)})\subset\mathcal{F}^{ss}_{\delta}(f^{N}(y),f)$, $\quad\forall y\in M$.
\noindent From Lemma~\ref{Lemma:Expansion for SH'} it follows that
\begin{equation*}
\prod^{n}_{j=0}\parallel Df^{-1}_{\mid E^{c}(f^{-j}(z))} \parallel<\alpha^{n},\quad 0\leq n\leq N,
\quad\forall z\in f^{N}(\overline{\mathcal{F}^{ss}_{\hat{r}}(h,f)}) 
\end{equation*}
and therefore
\begin{equation*}
\prod^{n}_{j=0}\parallel Df^{-1}_{\mid E^{cu}(f^{-j}(z))} \parallel<\alpha^{n},
\quad 0\leq n\leq N,\quad\forall z\in f^{N}(\overline{\mathcal{F}^{ss}_{\hat{r}}(h,f)}). 
\end{equation*}

\noindent Then by Lemma~\ref{Lemma:Backward iterations contract} we conclude that:

\begin{equation*} 
f^{-N}(W^{cu}_{r_{1}}(f^{N}(y),f))\subset f^{-N}(W^{cu}_{r_{0}}(f^{N}(y),f))\subset
W^{cu}_{\alpha^{N}_{1}r_{0}}(y,f)\subset W^{cu}_{\overline{\delta}}(y,f) 
\end{equation*}
for any  $y\in\overline{\mathcal{F}^{ss}_{\hat{r}}(h,f)}$.

\noindent Put $y_{1}=h$. We know that $\mathcal{F}^{ss}_{K}(f^{-n_{0}}(h),f)$
intersects $W^{cu}_{d}(f^{N}(h),f)$ in some point $z$. Then there exists
$y_{2}\in\mathcal{F}^{ss}_{\hat{r}}(h,f)$ such that $z=f^{-n_{0}}(y_{2})$
and a continuous function $\pi_{h,y_{2}}:W^{cu}_{\overline{\delta}}(h,f)\rightarrow W^{cu}_{\hat{\delta}}(y_{2},f)$
such that if $t'=\pi_{h,y_{2}}(t)$ then $t'\in\mathcal{F}^{ss}(t,f)$.

\noindent Set $x=f^{N}(h)$. Lemma~\ref{Lemma:Encapsulating cylinder} implies
the existence of a cylinder $\hat{C}$ containing 
\begin{equation*}
A=W^{cu}_{r_{1}}(x,f)\cup(\bigcup_{y\in W^{cu}_{d}(z,f)}\mathcal{F}^{ss}_{\epsilon_{1}}(y,f))
\end{equation*}
and such that for any $y\in W^{cu}_{d}(z,f)$ the intersection of the manifolds $\mathcal{F}^{ss}_{\epsilon_{1}}(y,f)$
and $W^{cu}_{r_{1}}(x,f)$ is exactly one point.

Now let $\phi:\hat{C}\rightarrow \hat{U}^{cu}\times\hat{V}^{ss}$ be the trivializing local
chart of the strong stable foliation with $0\in \phi(\hat{C})$ and
$\hat{\pi}:\hat{U}^{cu}\times \hat{V}^{ss}\rightarrow \hat{U}^{cu}\times\{0\}$ the
projection. Observe that if $\pi_{1}=\hat{\pi}_{\mid \phi(\overline{W^{cu}_{r_{1}}(x,f)})}$ then $\pi_{1}$
is a homeomorphism on its image.

On the other side if $\pi_{2}=\hat{\pi}_{\mid \phi(W^{cu}_{d}(z,f))}$,
as the intersection between the manifolds $\mathcal{F}^{ss}_{\epsilon_{1}}(y,f)$
and $W^{cu}_{r_{1}}(x,f)$ is exactly one point, it follows 
\begin{equation*}
\pi_{2}(\phi(W^{cu}_{d}(z,f)))\subset\pi_{1}(\phi(W^{cu}_{r_{1}}(x,f))). 
\end{equation*}

\noindent Then the function 

\begin{equation*}
g:\pi_{1}(\phi(\overline{W^{cu}_{r_{1}}(x,f)}))\rightarrow\pi_{2}(\phi(W^{cu}_{d}(z,f))) 
\end{equation*}
defined by
$g=\pi_{2}\circ\phi\circ f^{-n_{0}}\circ \pi_{h,y_{2}}\circ f^{-N}\circ (\phi)^{-1}\circ (\pi_{1})^{-1}$
is well defined and continuous. Hence by Brower's fixed point Theorem there exists a fixed point 
$p\in\pi_{1}(\phi(W^{cu}_{r_{1}}(x,f)))$, that is
\begin{equation}\label{The equation}
\pi_{2}\circ\phi\circ f^{-n_{0}}\circ \pi_{h,y_{2}}\circ f^{-N}\circ (\phi)^{-1}\circ (\pi_{1})^{-1}(p)=p.
\end{equation}

Observe that $\hat{\pi}(\pi^{-1}_{1}(p))=\hat{\pi}(\pi^{-1}_{2}(p))=p$
so if  $p_{r_{1}}=(\phi)^{-1}\circ\pi^{-1}_{1}(p)$,
$p_{-n_{0}}=(\phi)^{-1}\circ\pi^{-1}_{2}(p)$ then  $\mathcal{F}^{ss}(p_{r_{1}},f)=\mathcal{F}^{ss}(p_{-n_{0}},f)$ and
\begin{equation*}
\mathcal{F}^{ss}(f^{-N}(p_{-n_{0}}),f)=\mathcal{F}^{ss}(f^{-N}(p_{r_{1}}),f)\hspace{0.1 cm}\text{and}\hspace{0.1 cm}
\mathcal{F}^{ss}(f^{n_{0}}(p_{-n_{0}}),f)=\mathcal{F}^{ss}(f^{n_{0}}(p_{r_{1}}),f). 
\end{equation*}
From \eqref{The equation}
\begin{equation*}
\pi_{h,y_{2}}\circ f^{-N}(p_{r_{1}})=f^{n_{0}}(p_{-n_{0}}). 
\end{equation*}
So
\begin{align*}
\mathcal{F}^{ss}(f^{n_{0}}(p_{r_{1}}),f)&=\mathcal{F}^{ss}(f^{n_{0}}(p_{-n_{0}}),f)=
\mathcal{F}^{ss}(\pi_{h,y_{2}}\circ f^{-N}(p_{r_{1}}),f)\\
\hspace{0.5 cm}&=\mathcal{F}^{ss}(f^{-N}(p_{r_{1}}),f)=\mathcal{F}^{ss}(f^{-N}(p_{-n_{0}}),f)
\end{align*}
which implies $\mathcal{F}^{ss}(p_{r_{1}},f)=\mathcal{F}^{ss}(f^{-n_{0}-N}(p_{-n_{0}}),f)$.	

Observe that $p_{-n_{0}}\in W^{cu}_{d}(z,f)$, $p_{r_{1}}\in W^{cu}_{r_{1}}(x,f)=W^{cu}_{r_{1}}(f^{N}(h),f)$ and
$p_{r_{1}}\in\mathcal{F}^{ss}_{\epsilon_{1}}(p_{-n_{0}},f)$.
Thus $dist_{\mathcal{F}^{ss}}(p_{r_{1}},p_{-n_{0}})\leq\epsilon_{1}$
and $p_{-n_{0}}=f^{-n_{0}}\circ\pi_{h,y_{2}}\circ f^{-N}(p_{r_{1}})\in f^{-n_{0}}(W^{cu}_{\hat{\delta}}(y_{2},f))$.\\
\noindent Take $\theta\in\mathcal{F}^{ss}_{\delta}(p_{r_{1}},f)$ arbitrary. Then
\begin{equation*}
dist_{\mathcal{F}^{ss}}(\theta,p_{-n_{0}})\leq dist_{\mathcal{F}^{ss}}(\theta,p_{r_{1}})+
dist_{\mathcal{F}^{ss}}(p_{r_{1}},p_{-n_{0}})\leq \delta+\epsilon_{1}<K 
\end{equation*}
and from there
\begin{equation*}
\mathcal{F}^{ss}_{\delta}(p_{r_{1}},f)\subset\mathcal{F}^{ss}_{K}(p_{-n_{0}},f)\subset
f^{-n_{0}}(\mathcal{F}^{ss}_{\hat{r}}(f^{n_{0}}(p_{-n_{0}}),f)). 
\end{equation*}
Remember that
\begin{align*}
f^{N}(\overline{\mathcal{F}^{ss}_{\hat{r}}(f^{n_{0}}(p_{-n_{0}}),f)})&\subset f^{N}(\overline{\mathcal{F}^{ss}_{2\hat{r}}(f^{-N}(p_{r_{1}}),f)})\subset
\mathcal{F}^{ss}_{\delta}(f^{N}(f^{-N}(p_{r_{1}})),f)\\
\hspace{0.5 cm}&=\mathcal{F}^{ss}_{\delta}(p_{r_{1}},f)\subset f^{-n_{0}}(\mathcal{F}^{ss}_{\hat{r}}(f^{n_{0}}(p_{-n_{0}}),f))
\end{align*}

\noindent From there
\begin{align*}
f^{n_{0}+N}(f^{-n_{0}}(\overline{\mathcal{F}^{ss}_{\hat{r}}(f^{n_{0}}(p_{-n_{0}}),f)})&=
f^{N}(\overline{\mathcal{F}^{ss}_{\hat{r}}(f^{n_{0}}(p_{-n_{0}}),f)})\\
\hspace{0.5 cm}&\subset f^{-n_{0}}(\overline{\mathcal{F}^{ss}_{\hat{r}}(f^{n_{0}}(p_{-n_{0}}),f)}).
\end{align*}
Again by Brower's fixed point Theorem there exists
$Q\in f^{-n_{0}}(\overline{\mathcal{F}^{ss}_{\hat{r}}(f^{n_{0}}(p_{-n_{0}}),f)})
\subset f^{-n_{0}}(C)$
a fixed point by the function $f^{n_{0}+N}$, and hence $f^{n_{0}}(Q)$ a periodic point in $C$.

\end{proof}

%% file: Examples.tex
\section{Examples}
\label{sec:Examples}

\subsection{Shub's example}
\label{subsec:Shub's example} \quad
\newline
The conditions in Corollary~\ref{Corollary:Observacion 17 del Referee} and
Theorem~\ref{Theorem:Density of periodic points} 
are fulfilled by the widely known example of Shub.
This is because Pujals-Sambarino proved in ~\cite{PS}
that it satisfies the Property SH and that its stable foliation
is robustly minimal.

\subsection{A wider scenario for SH on the inverse}
\label{subsec:A wider scenario for SH on the inverse} \quad
\newline

\noindent The following Proposition~\ref{Proposition:Diffeomorphisms whose inverse is SH}
show how, under some conditions, to get perturbations with the Property SH
and whose inverses also has the Property SH. All the conditions in this Proposition
are fulfilled, in particular, by Ma\~{n}\'e's example.

Let $M$ be a smooth compact boundaryless three dimensional manifold and
$\mathcal{T}$ the set of non Anosov robustly transitive partially hyperbolic
diffeomorphisms in $M$. Denote by $\mathcal{T}'$ the subset of $\mathcal{T}$ consisting of
the diffeomorphisms with strong stable robustly minimal foliation.

\begin{Proposition}\label{Proposition:Diffeomorphisms whose inverse is SH}
There exists an open and dense subset $\mathcal{D}'$ of $\mathcal{T}'$
such that for every $g\in\mathcal{D}'$ we have that $g^{-1}$ satisfies
Property SH.
\end{Proposition}
\begin{proof}
Just apply the following Claim~\ref{Claim:Open dense subset with blenders}  and
Proposition~\ref{Proposition:Enrique}. 
\end{proof}

\begin{Claim}\label{Claim:Mane}
Let $f\in Diff^{1}(M)$ be a diffeomorphism such that its periodic points are
$C^{1}$-robustly hyperbolic and $\Omega(f)=M$. Then $f$ is Anosov.
\end{Claim}

\begin{proof}

See ~\cite{M1}.

\end{proof}

\begin{Claim}\label{Claim:Hyperbolic periodic points with different indices}
There is a dense subset $\mathcal{A}$ of $\mathcal{T}$
such that for every $f\in\mathcal{A}$ 
there exists a pair  of hyperbolic periodic points with different indices.
\end{Claim}

\begin{proof}
It follows from Claim~\ref{Claim:Mane} and ~\cite{M2}. 
\end{proof}

\begin{Claim}\label{Claim:Dense subset with heterodimensional cycles}
There exists a dense subset $\mathcal{B}$ of $\mathcal{T}$ such that every 
$f\in\mathcal{B}$ has a heterodimensional cycle of codimension one.
\end{Claim}
\begin{proof}
It follows from Claim~\ref{Claim:Hyperbolic periodic points with different indices}
and ~\cite{BDPR}. 
\end{proof}

\begin{Claim}\label{Claim:Open dense subset with blenders}
There exists an open and dense subset $\mathcal{D}$ of $\mathcal{T}$ 
such that every $g\in\mathcal{D}$ has a cs\emph{-blender} and a cu\emph{-blender}.   
\end{Claim}

\begin{proof}
It follows from Claim~\ref{Claim:Dense subset with heterodimensional cycles}
and Proposition~\ref{Proposition:Bonatti, Diaz, Viana}.
\end{proof}